\documentclass[onecolumn]{autart}    

\usepackage{graphicx}           
\usepackage{dcolumn}            
\usepackage{bm}                 

\usepackage{graphics}           
\usepackage{epsfig}             
\usepackage{times}              
\usepackage[fleqn]{amsmath}
\usepackage{amssymb}
\usepackage{extarrows}

\usepackage{ntheorem} 

\usepackage{amsfonts}
\usepackage{fancyhdr}
\usepackage{color}
\usepackage{subfigure}
\usepackage{enumerate}
\usepackage{url}

\allowdisplaybreaks[4] 
\newtheorem{theorem}{Theorem} 
\newtheorem{lemma}{Lemma}

\newtheorem{proposition}{Proposition}

\newtheorem{remark}{Remark}

\newtheorem*{proof}{Proof}

\newcommand{\diff}[1]{\text{d}{#1}}

\begin{document}

\begin{frontmatter}

\title{Aggregate Power Control of Heterogeneous TCL Populations Governed by Fokker-Planck Equations} 

\author{Jun Zheng$^{1}$}\ead{zhengjun@swjtu.edu.cn},
\author{Gabriel Laparra$^{2}$}\ead{gabriel.laparra@polymtl.ca},
\author{Guchuan~Zhu$^{2}$}\ead{guchuan.zhu@polymtl.ca},
\author{Meng~Lia$^{3}$}\ead{lmbuaa@gmail.com}

\address{$^{1}${School of Mathematics, Southwest Jiaotong University,
        Chengdu, Sichuan, P. R. of China 611756}\\
        $^{2}$Department of Electrical Engineering, Polytechnique Montr\'{e}al,  P.~O. Box 6079, Station Centre-Ville, Montreal, QC, Canada H3T 1J4 \\
        $^{3}$School of Information Science and Technology, Zhejiang Sci-Tech University, Zhejiang, China}

\begin{abstract}                           
This paper addresses the modeling and control of heterogeneous populations of thermostatically controlled loads (TCLs) operated by model predictive control (MPC) schemes at level of each TCL. It is shown that the dynamics of such TCLs populations can be described by a pair of Fokker-Planck equations coupled via the actions in the domain. The technique of input-output feedback linearization is used in the design of aggregate power control, which leads to a nonlinear system in closed loop. Well-posedness analysis is carried out to validate the developed control scheme. The closed-loop stability of the system is assessed by a rigorous analysis. A simulation study is contacted, and the obtained results confirm the validity and the effectiveness of the proposed approach.
\end{abstract}

\end{frontmatter}
%
\section{Introduction}\label{Sec: Introduction}
A{ggregates} of large populations of thermostatically controlled loads (TCLs) can be managed to offer auxiliary services, such as frequency control, load following, and energy balancing, which can contribute to maintaining the overall stability of power networks \cite{Callaway:2009,Ledva2018Managing,Liu:2016,Liu2016MPC}. TCLs can also provide a means for absorbing the fluctuations of renewable energy generated by, e.g., wind turbines and solar photovoltaic plants \cite{Callaway:2009,Mahdavi:2017}. Moreover, due to the fact that most of the TCLs, including space heaters, air conditioners, hot water tanks, and refrigerators, exhibit flexibilities in power demand for their operation and elasticities in terms of performance restrictions, they are considered to be one of the most important Demand Response (DR) resources that can provide such features as power peak shaving and valley filling and enable dynamic pricing schemes in the context of the Smart Grid \cite{Angeli:2012,brandstetter2017hierarchical,collins2019distributed,gu2017online,ma2016_IEEE_TSG,perez2016integrated,Qi2017,sadid2017discrete,tindemans2015decentralized,Totu:2017,7000567,Zhang2013Aggregated}. Indeed, control of aggregated TCL populations is a long-time standing problem, which continues to attract much attention in the recent literature {\cite{Liu:2016,Kizilkale2019mfg,Voice:2018,Zhao:2017}}.

The aim of this paper is to develop strategies for aggregate power control of heterogeneous TCL populations {}{based on continuum models} described by partial differential equations (PDEs).  The PDE aggregate model of TCL populations was originally introduced in \cite{Malhame:1985}. Under the assumption of homogenous TCL populations with all the TCLs modeled by thermostat-controlled scalar stochastic differential equations, the dynamics of a population are expressed by two coupled Fokker-Planck equations describing the evolution of the distribution of the TCLs in ON and OFF states, respectively, over a rang of temperature. In about a quarter century later, this PDE aggregate model has been elaborated in the work~\cite{Callaway:2009} to incorporate control actions to manipulate the aggregate of TCL populations, such as the total power consumption. This approach is adopted later by, e.g., {\cite{Angeli:2012,tindemans2015decentralized,Totu:2017,Beeker2016fpe}} for different DR applications. A more generic stochastic hybrid system model applicable to a wider class of responsive loads is developed in~\cite{Zhao:2018}. PDE aggregate models can also be derived in the framework of deterministic systems. Inspired by the concept of fluid-flow dynamics, a semi-linear transport PDE has been derived in \cite{Bashash2013} for homogeneous TCL populations. It is proposed later in \cite{moura2013modeling} that this PDE model can be extended to describe heterogeneous TCL populations by adding a diffusive term, which indeed results in the same type of PDEs as the one developed in the framework of stochastic systems, i.e., a pair of coupled Fokker-Planck equations. Moreover, a heterogenous population with high diversity can be divided into a finite number of subsets, each of which represents a population with limited variation in its heterogeneity~\cite{Ghaffari:2015}.

Another widely adopted model for describing the dynamics of TCL populations is the one originally proposed in \cite{Koch2011Modeling}. In this model, a  temperature deadband is divided into a finite number of segments, called state bins, and the evolution of the distribution of TCLs over the temperature under deadband control is presented by a linear system for which the state transactions can be derived based on Markov chains. This model has been further extended to TCLs of more generic dynamics, e.g., higher order systems, with different control mechanisms for different applications {\cite{Ledva2018Managing,Liu:2016,Liu2016MPC,Zhang2013Aggregated,Soudjani2015,Hao2015Aggregated}}.

In this work, we consider the TCL populations for which the individual TCLs are operated under model predictive control (MPC), which is one of the most frequently adopted techniques in practice due to its ability to handle constraints, time varying processes, delays, and uncertainties, as well as disturbances. It should be noted that as with MPC schemes, TCL switching may occur arbitrarily over the time at any point inside a range of temperature. Consequently, the state-bin model based on the transaction of Markov chains is not applicable. {}{For this reason and being inspired by the work reported in the recent literature, we concentrate on PDE-based aggregate modeling and control of large TCL populations in this paper.}

The approach adopted in most of the work on PDE-based aggregate power control of TCL populations is based on the technique of \emph{early-lumping} with which the control design is carried out on first discretizing the PDE models over the space {\cite{Callaway:2009,Totu:2017,Bashash2013,Ghaffari:2015,le2016pde}}. There are few exceptions, such as the control scheme for output regulation of TCL populations proposed in \cite{Ghanavati:2018}, which are based on the approach of \emph{late-lumping} and hence, the original PDE model is used in control design without approximations. One of the main advantages of this approach is that it can preserve the basic property of the original system, in particular the closed-loop stability, when the designed control is applied \cite{Balas:1978}.

The control design presented in this work adopts the approach of \emph{late-lumping}, in which the original PDE model is used. Specifically, the technique of input-output linearization {}{(see, e.g., \cite{Christofides1996,Christofides2001,Maidi:2014}), which is an extension of the same technique for finite dimensional nonlinear system control \cite{Khalil:2002,Krstic:1995,Isidori:1995,Levine:2009}}, is used in the design of aggregate power (the output) regulation of a TCL population governed by a pair of coupled Fokker-Planck equations. It should be pointed out that this method is greatly inspired by the work presented in \cite{Maidi:2014} while being different in that the control design is carried out directly with the original Fokker-Planck equations. Consequently, a guarantee on closed-loop stability is of primary importance from both theoretical and practical perspectives. To assess the validity of the developed control scheme, an analysis on the well-posedness and the stability of the closed-loop system {}{with nonlinear nonlocal terms is performed. It is worth noting that to the best of our knowledge, the result obtained in this paper is novel for PDE-based aggregate power control of large TCL populations. Particularly, the results on the well-posedness and the stability of Fokker-Planck equations having nonlinear nonlocal terms distributed both in the domain and on the boundary and involving integration in the denominator of nonlocal terms are novel.}

{}{The main contributions of this paper include:}
\begin{itemize}
\item {}{incorporating MPC-based control schemes at the TCL level into an aggregate model of large TCL populations described by a pair of coupled Fokker-Planck equations;}
\item {}{developing an aggregated power control scheme based on the technique of input-output linearization for PDE systems;}
\item {}{assessing the well-posedness of the solutions of the closed-loop system with nonlinear nonlocal terms by the technique of iteration and the generalization of the approach developed in \cite{Ladyzenskaja:1968};}
\item {}{conducting a rigorous stability analysis of the closed-loop system composed of finite dimensional input-output dynamics and infinite dimensional internal dynamics.}
\end{itemize}

In the rest of this paper, Section~\ref{Sec: Modeling} presents the dynamical model of individual TCLs under MPC and the derivation of aggregate model of homogeneous and heterogeneous TCLs populations. Control design is given in Section~\ref{Sec: Control design}. Wellposedness
assessment {}{by the technique of iteration} and stability analysis are detailed in
Section \ref{Sec: Well-posedness} and Section \ref{Sec: Stability Analysis}, respectively. A simulation study to evaluate the behavior of the closed-loop system with the developed control scheme is presented in Section~\ref{Sec: Simulation Study}. {}{Some} concluding remarks are provided in Section~\ref{Sec: Conclusion}. Finally, details on some technical development are presented in {}{the appendices}.

\textbf{Notations:} If there is no ambiguity, we use the notations {$\dfrac{\partial ^m u}{\partial x^m}$, $D^m_xu$}, or $u_{\underbrace{x\cdots x}_{m~\text{times}}}$, or $u^{(m)}$ to denote the $m$th-order derivative (or weak derivative) of a function $u$ w.r.t. its argument $x$, where $m$ is a nonnegative integer.

For $T\in (0,+\infty)$, let $Q_T=(\underline{x},\overline{x})\times(0,T)$ with $ \overline{Q}_T=[\underline{x},\overline{x}]\times[0,T]$. {{}Denote $Q_\infty=(\underline{x},\overline{x})\times(0,\infty)$ and} $\overline{Q}_\infty=[\underline{x},\overline{x}]\times[0,\infty)$. Let $l$ be a nonintegral positive number. The spaces $\mathcal {H}^{l}([\underline{x},\overline{x}])$ with the norm $|\cdot |_{(\underline{x},\overline{x})}^{(l)}$ and $\mathcal {H}^{l,\frac{l}{2}}(\overline{Q}_T)$ with the norm $|\cdot |_{Q_T}^{(l)}$ are defined in {{} Appendix~\ref{Appendix: notation}}.

For $T\in (0,+\infty]$, the space $L^{\infty}((0,T);L^\infty(\underline{x},\overline{x}))$ consists of all strongly measurable functions $u:{}{(0,T)}\rightarrow L^\infty(\underline{x},\overline{x})$ with the norm
\begin{align*}
\|u\|_{\infty} {}{:=}\|u\|_{L^{\infty}((0,T);L^\infty(\underline{x},\overline{x}))} {}{:=}\text{vrai}\sup\limits_{0< s<T}\|u(\cdot,s)\|_{L^{\infty}(\underline{x},\overline{x})}{}{:=}\text{vrai}\sup\limits_{0< s<T}\text{vrai}\sup\limits_{\underline{x}< x< \overline{x}}|u(x,s)|<+\infty.
\end{align*}

Finally, for simplicity, we denote by $\|u\|_p$ the norm $ \|u\|_{L^p(\underline{x}, \overline{x})}$ for a function $u\in L^{p}(\underline{x}, \overline{x})$, where $1\leq p\leq +\infty$. Particularly, we use $\|u\|$ to denote $\|u\|_2$.

\section{Modeling Aggregate of TCL Populations}\label{Sec: Modeling}
\subsection{Dynamics and Control of Individual TCLs}\label{Sec: Individual Dynamics of TCLs}
We consider a large population of thermostatically controlled loads (TCLs) operated in ON/OFF mode. Denote by $x_i$ the temperature of the $i$-th load in a population of size $N$. {}{The dynamics of $x_i$ are described by a first-order ordinary differential equation (ODE):}
\begin{equation}\label{eq: thermal dynamics}
  \dfrac{\text{d}x_i}{\text{d}t} = \dfrac{1}{R_i C_i}\left(x_{ie}-x_i + s_iu_iR_iP_i\right), \ i=1,\ldots,N,
\end{equation}
where $x_{ie}$ is the external temperature, $R_i$ is thermal resistance, $C_i$ is thermal capacitance, $P_i$ is thermal power, {}{$s_i$ is defined as
\begin{equation*}
  s_i =
  \begin{cases}
    1, & \hbox{heating system;} \\
    -1, & \hbox{cooling system,}
  \end{cases}
\end{equation*}
and the switching signal $u_i$:
\begin{equation*}
  u_i =
  \begin{cases}
    1, & \hbox{ON;} \\
    0, & \hbox{OFF,}
  \end{cases}
\end{equation*}
is generated by an MPC scheme specified below. Note that this hybrid dynamic model is widely adopted in the literature, while deadband-based switching is the most used control scheme in different works (see, e.g., \cite{Callaway:2009,Mahdavi:2017,Angeli:2012,tindemans2015decentralized,Totu:2017,Bashash2013,moura2013modeling,Ghaffari:2015}).}

The model used in MPC design is a discrete-time system that can be derived from~\eqref{eq: thermal dynamics} by using the standard scheme of zero-order hold in a sampling period $T_{s}$, which takes the form:
\begin{equation}\label{Eq: Time discrete}
x_i(k+1)=A_i x_i(k) + B_i u_i(k) + E_i x_{ie}(k),
\end{equation}
where the system parameters are given by $ A_i=\text{e}^{-\frac{1}{R_iC_i}T_{s}}$, $B_i= s_i\frac{P_i}{C_i}\int_{0}^{Ts} \text{e}^{-\frac{1}{R_iC_i}\tau}\text{d}\tau$, and $E_i= \frac{1}{R_iC_i}\int_{0}^{Ts} \text{e}^{-\frac{1}{R_iC_i}\tau}\text{d}\tau$.


Let $x_{\text{ref}}$ be the reference signal. Denote by $e_i(k) = x_i(k)-x_{\text{ref}}(k)$ the tracking error. Then for an output regulation problem corresponding to the $i$-th subsystem the MPC design is an optimization problem expressed by \cite{saad2018discrete}:
\begin{subequations}\label{eq: MPC}
\begin{align}
  &\min_{u_i(1),\ldots,u_i(M)} \sum\limits_{k=1}^{M} \left(\|e_i(k)\|^2_{Q_{\text{mpc}}} + \|u_i(k)\|^2_{R_{\text{mpc}}}\right), \label{eq: MPC obj}\\
  \text{s.t.:} \ \ & x_i(k+1)=A_i x_i(k) + B_i u_i(k) + E_i x_{ie}(k), \label{eq: MPC constrant1}\\
                   & {}{x_i(t) \rightarrow x_i(0), \ u_i(t) \rightarrow u_i(0), \ t > 0,} \label{eq: MPC starting point}\\
      & u_i (k)   \in \{0,1\}, \label{eq: MPC constrant2}\\
      & x_{\text{\text{ref}}}(k)-\underline{\Delta x} \le  x_i (k) \le x_{\text{\text{ref}}}(k)+ \overline{\Delta x},\label{eq: MPC constrant3}
\end{align}
\end{subequations}
where $M$ is the prediction horizon, $Q_{\text{mpc}}$ and $R_{\text{mpc}}$ are, respectively, the weighting factors applied to penalize the tracking errors and the control efforts, and $\underline{\Delta x}$ and $\overline{\Delta x}$ are, respectively, the lower and upper temperature bounds. Switching sequences are determined by the solution of the above problem. At each time step of running, the controller compute the solutions until $P$ steps, while only the first one is applied. This process is repeated at the next sampling time, constituting a moving horizon scheme. {}{Note that in the implementation of MPC, the control sequence is applied to the continuous-time system~\eqref{eq: thermal dynamics} through a mechanism of zero-order hold. Consequently, the hybrid nature of TCLs with their dynamics governed by ODEs while switching between ON and OFF states will be preserved.}
%


\subsection{{}{PDE Model for Heterogeneous Populations of TCLs}}\label{Sec: Sec: Diffusive Model}
Let $v(x,t)$ and $w(x,t)$ be the distribution of loads [number of loads/$^{\circ}$C] at temperature $x$ and time $t$, over the OFF and ON states, respectively. For notational simplicity, we consider hereafter only the setting for cooling systems. The models for the setting of heating systems can be established in the same way.

First, consider the distribution of loads over the ON state. Let $F(x,t)$ denote the flow of loads at the point $(x,t)$. Then, {}{based on the model describing the microscopic dynamics of individual TCLs given in Section~\ref{Sec: Individual Dynamics of TCLs}}, we have for a population of homogeneous loads:
\begin{align}
  & F(x,t) = \alpha_{1}(x) w(x,t), \label{eq: ON flow in}
\end{align}
where
\begin{equation}\label{eq: alpha ON}
  \alpha_{1}(x) = \dfrac{1}{R C}\left(x_{e}-x - RP\times 1\right),
\end{equation}
with $x_{e}$, $R$, $C$, and $P$ representing the external temperature, thermal resistance, capacitance, and power, respectively, of the population.

{}{
Let $\delta_{0\rightarrow 1}(x,t)$ denote the added flow due to the switch of loads from OFF state to ON state at $(x,t)$, which is induced by the operation of MPCs at the TCL level. Note that as pointed out in \cite{moura2013modeling}, the aggregate power consumption of a heterogenous TCL population exhibits a damped response to step set-point changes. This property can be represented by the diffusivity added to the macroscopic fluid-flow TCL population model (see, e..g. \cite{Bashash2013}), which leads to the following Fokker-Planck equation:
\begin{align}
  \dfrac{\partial w}{\partial t}(x,t) =& \dfrac{\partial}{\partial x}\left( \beta\dfrac{\partial w}{\partial x}(x,t)- \alpha_{1}(x)w(x,t)\right)
                                         + \delta_{0 \rightarrow 1}(x,t), {{}\; (x,t)\in Q_{\infty},} \label{eq: FPE ON 0}
\end{align}
where $\beta$ is the diffusion coefficient.}

{}{Similarly, we can derive the dynamic model of the distribution of loads over the OFF state, which is given by
\begin{align}
  \dfrac{\partial v}{\partial t}(x,t) =& \dfrac{\partial}{\partial x}\left( \beta\dfrac{\partial v}{\partial x}(x,t)- \alpha_{0}(x)v(x,t)\right)
                                          {}{+\delta_{1 \rightarrow 0}(x,t)}, {{}\; (x,t)\in Q_{\infty},} \label{eq: FPE OFF 0}
\end{align}
where
\begin{equation}\label{eq: alpha OFF}
  \alpha_{0}(x) = \dfrac{1}{R C}\left(x_{e}-x\right),
\end{equation}
and $\delta_{1 \rightarrow 0}(x,t)$ denotes the added flow due to the switch of loads from ON-state to OFF-state at $(x,t)$.}

Suppose that in the considered problem the operation is mass-conservative, that is the size of the population is constant. Indeed, it is reasonable to assume that no TCLs will be added to or removed from the population during a sufficiently long period of operation. Then, the flow due to the switch of loads at $(x,t)$ must verify the compatibility condition $\delta_{0 \rightarrow 1}(x,t) = -\delta_{1 \rightarrow 0}(x,t) := \delta(x,t)$. Therefore, the dynamics of the evolution of the distribution of loads can be expressed by the following coupled Fokker-Planck equations:
\begin{subequations}\label{eq: FPE}
\begin{align}
  \dfrac{\partial w}{\partial t}(x,t) =& \dfrac{\partial}{\partial x}\left( \beta\dfrac{\partial w}{\partial x}(x,t)- \alpha_{1}(x)w(x,t)\right)
                                          +\delta(x,t), {{}\; (x,t)\in Q_{\infty},} \label{eq: FPE ON}\\
  \dfrac{\partial v}{\partial t}(x,t) =& \dfrac{\partial}{\partial x}\left( \beta\dfrac{\partial v}{\partial x}(x,t)- \alpha_{0}(x)v(x,t)\right)
                                          -\delta(x,t), {{}\; (x,t)\in Q_{\infty}}. \label{eq: FPE OFF}
\end{align}
\end{subequations}
{}{To assure the validity of the PDE aggregate model given in~\eqref{eq: FPE}, we impose the following boundary and initial conditions}
\begin{subequations}\label{eq: BIs of FPE}
\begin{align}
  &\beta\dfrac{\partial w}{\partial x}(\underline{x},t)- \alpha_{1}(\underline{x})w(\underline{x},t)=0,\\
   &\beta\dfrac{\partial w}{\partial x}(\overline{x},t)- \alpha_{1}(\overline{x})w(\overline{x},t)=0,\\
  &\beta\dfrac{\partial v}{\partial x}(\underline{x},t)- \alpha_{0}(\underline{x})v(\underline{x},t)=0,\\
   &\beta\dfrac{\partial v}{\partial x}(\overline{x},t)- \alpha_{0}(\overline{x})v(\overline{x},t)=0,\\
  &w(x,0) = w^0, \ v(x,0) = v^0,\label{eq: ICs of FPE}
\end{align}
\end{subequations}
{}{which guarantees the conservativity of mass of the problem under consideration, as stated below.}
\begin{proposition} Let the size of the population modeled by \eqref{eq: FPE} and \eqref{eq: BIs of FPE}  be defined as:
\begin{equation*}
  N_{\text{agg}}(t) = \int_{\underline{x}}^{\overline{x}} \left(w(x,t)+v(x,t)\right)\text{d}x, \ t\geq 0.
\end{equation*}
We have then
\begin{equation*}
  N_{\text{agg}}(t) = N_{\text{agg}}(0) = \int_{\underline{x}}^{\overline{x}} \left(w^0(x)+v^0(x)\right)\text{d}x, \ \forall t \geq 0.
\end{equation*}
\end{proposition}
{}{\begin{proof}
The result can be obtained directly by taking integration of \eqref{eq: FPE} and using the boundary conditions \eqref{eq: BIs of FPE}.
\end{proof}}

\begin{remark}
In the formulation proposed in this work, the reference temperature is the same for all the TCLs in the population. However, it is important to note that the range around the reference which each TCL attempts to restrict to (also called the comfort zone) can be different from other TCLs depending on its performance requirement and power consumption consideration. This feature is significantly different from that in thermostat-based deadband control schemes where the turn-on and turn-off temperatures are identical for all the TCLs in the same population.
\end{remark}

For notational simplicity, we {}{may} ignore hereafter the domain on which the system is defined and the arguments of different functions if there is no ambiguity.

\subsection{Controlled PDE Model for Power Consumption Control}\label{Sec: Forced Diffusive Model}
As mentioned earlier, power consumption control of a TCL population can be achieved by moving the control volume with the set-point temperature {\cite{Callaway:2009,Bashash2013,Ghaffari:2015}}. By taking the time-derivative of the set-point as the control input: $u(t) = \dot{x}_{\text{ref}}$, the dynamics of the evolution of the distribution of a population of heterogenous loads can be expressed by the following forced Fokker-Planck equations:
\begin{subequations}\label{eq: foced FPE}
\begin{align}
  \dfrac{\partial w}{\partial t}(x,t) =& \dfrac{\partial}{\partial x}\left( \beta\dfrac{\partial w}{\partial x}(x,t)- (\alpha_{1}(x)-u(t))w(x,t) \right) +\delta(x,t), {{}\; (x,t)\in Q_{\infty},} \label{eq: foced FPE ON}\\
  \dfrac{\partial v}{\partial t}(x,t) =&\dfrac{\partial}{\partial x}\left( \beta\dfrac{\partial v}{\partial x}(x,t)- (\alpha_{0}(x)-u(t))v(x,t) \right) -\delta(x,t), {{}\; (x,t)\in Q_{\infty},} \label{eq: foced FPE OFF}
\end{align}
\end{subequations}
with the following boundary and initial conditions:\begin{subequations}\label{eq: BCs of FPE}
\begin{align}
  &\beta\dfrac{\partial w}{\partial x}(\underline{x},t)- (\alpha_{1}(\underline{x})-u(t))w(\underline{x},t)=0,\\
   &\beta\dfrac{\partial w}{\partial x}(\overline{x},t)- (\alpha_{1}(\overline{x})-u(t))w(\overline{x},t)=0,\\
  &\beta\dfrac{\partial v}{\partial x}(\underline{x},t)- (\alpha_{0}(\underline{x})-u(t))v(\underline{x},t)=0,\\
   &\beta\dfrac{\partial v}{\partial x}(\overline{x},t)- (\alpha_{0}(\overline{x})-u(t))v(\overline{x},t)=0,\\
  &w(x,0) = w^0, \ v(x,0) = v^0.
\end{align}
\end{subequations}

{}{For the validity of the PDE aggregate model given in \eqref{eq: foced FPE} and \eqref{eq: BCs of FPE}, it is expected that the closed-loop system preserves the property of mass conservativity, which is guaranteed by Proposition~\ref{prop. 8} given later in Section~\ref{Sec: Well-posedness}.
}

\section{Control Design}\label{Sec: Control design}
For the purpose of controlling the power consumption of the whole population of TCLs, we take the weighted total power consumption as the system output:
\begin{equation}\label{eq: output}
  y(t) = \dfrac{P}{\eta}\int_{\underline{x}}^{\overline{x}} (ax+b)w(x,t)\text{d}x,
\end{equation}
where $a$ and $b$ are constants with $a\neq 0$. {}{The purpose of introducing the weighting function $ax+b$ in system output defined above is to guarantee that the input-output dynamics of the system are well-defined in terms of characteristic index that is an analogue of the concept of relative degree for finite-dimensional systems \cite{Christofides1996,Christofides2001}}. The reason to use a specific weighting function in \eqref{eq: output} is mainly due to the fact that in general it is very difficult, if not impossible, to conduct well-posedness and stability analysis for nonlinear PDEs with generic nonlinear terms. Note that the weighting function is not unique, and it can be chosen and tuned in control design and implementation depending on control objectives and performance requirements.

We consider in the present work a set-point control problem where the objective is to drive the power consumption of the whole population to a desired level $y_d$ representing the demand from, e.g., energy suppliers.

Denoting by $e(t) = y(t) - y_d$ the regulation error, {}{we can derive from \eqref{eq: foced FPE ON} and \eqref{eq: output} that}
\begin{align*}
  \dfrac{\diff{e}}{\diff{t}}
  =& \dfrac{P}{\eta}\int_{\underline{x}}^{\overline{x}} (ax+b)\dfrac{\partial w}{\partial t}\text{d}x\\
  =& \dfrac{P}{\eta}\int_{\underline{x}}^{\overline{x}} (ax+b)\dfrac{\partial}{\partial x}\bigg( \beta\dfrac{\partial w}{\partial x}-
     (\alpha_{1}(x)-u(t))w \bigg)\text{d}x+ \dfrac{P}{\eta}\int_{\underline{x}}^{\overline{x}} (ax+b)\delta(x,t) \text{d}x.
\end{align*}
Performing integration by parts and applying the boundary conditions \eqref{eq: BCs of FPE}, we get
\begin{align}
  \dfrac{\diff{e}}{\diff{t}}
     =-&\dfrac{P}{\eta}\int_{\underline{x}}^{\overline{x}}a\bigg(\beta\dfrac{\partial w}{\partial x}-(\alpha_{1}(x)-u(t))w\bigg) \text{d}x+ \dfrac{P}{\eta}\int_{\underline{x}}^{\overline{x}} (ax+b)\delta(x,t) \text{d}x.\notag
\end{align}
Let
\begin{align}
   u(t) = &- \frac{\displaystyle\int_{\underline{x}}^{\overline{x}}\bigg( \beta\dfrac{\partial w}{\partial x}- \alpha_{1}(x)w \bigg)
            \diff{x} +\frac{\eta}{{}{aP}}\phi(t)}{\displaystyle\int_{\underline{x}}^{\overline{x}}{}{w}\diff{x}} \nonumber\\
        = &- \frac{\beta(w(\overline{x},t)-w(\underline{x},t))-\displaystyle\int_{\underline{x}}^{\overline{x}}\alpha_{1}(x)w
            \diff{x} +\frac{\eta}{{}{aP}}\phi(t)}{\displaystyle\int_{\underline{x}}^{\overline{x}}{}{w}\diff{x}}, \label{eq: control u}\\
   \Gamma(t)=& \dfrac{P}{\eta}\int_{\underline{x}}^{\overline{x}} (ax+b)\delta(x,t)\text{d}x. \label{eq: control Gamma}
\end{align}
Then the regulation error dynamics become:
\begin{equation}\label{eq: error dynamics 1}
  \dfrac{\diff{e}}{\diff{t}} = \phi(t) + \Gamma(t), \ \ e(0) = e_0,
\end{equation}
where $\phi(t)$ is an auxiliary control, and $e_0$ is the initial regulation error.

\begin{remark}
It can be seen from \eqref{eq: control u} that the characteristic index of the input-output dynamics of the system is 1 if at any time the power consumption of the whole population is not null. This constraint is a nature property for a TCL population of sufficiently large size and can be guaranteed by a simple operation in practice.
\end{remark}

\begin{remark}
{}{Note that the term $\Gamma(t)$ represents the effect of switching induced by the MPCs at the TCL level. An advantage of treating $\Gamma(t)$ as a disturbance in the proposed aggregate power control scheme is that the TCLs do not need to signal the instantaneous switching operations, which will allow greatly simplifying the implementation and considerably relaxing the performance requirement on communication systems.} Moreover, the developed control algorithm is a partial state-feedback scheme depending only on $w(x,t)$, which is exactly the same information needed for computing the system output {\cite{Callaway:2009,Totu:2017,Bashash2013,Ghaffari:2015, le2016pde, Ghanavati:2018}}. Finally, communicating the current operational
state and the output (the temperature) of the TCLs can be enabled by, e.g., advanced metering infrastructures in the context of the Smart Grid \cite{Callaway:2009}.
\end{remark}

The control synthesis will then be completed by finding a $\phi(t)$ that {}{can robustly stabilize the regulation error dynamics in the presence of disturbances.} For this aim, we apply a standard and very effective method, called nonlinear damping (see, e.g., \cite{Khalil:2002,Krstic:1995}). Specifically, it is reasonable to assume that $\Gamma(t)$ is uniformly bounded, i.e., $|\Gamma(t)| \leq \Gamma_{\infty}$ for all $t > 0$ with $\Gamma_{\infty}$ a positive constant, because at any moment only a limited number of TCLs change their operational state. Then, it is obvious that the control
\begin{equation}\label{eq: stabilizer}
  \phi(t) = -k_0e(t)
\end{equation}
with any constant gain $k_0 > 0$ will globally exponentially stabilize \eqref{eq: error dynamics 1} when $\Gamma(t) \equiv 0$ for all $t > 0$. Furthermore, $V = e^2$ is a control Lyapunov function for the disturbance-free counterpart of the system~\eqref{eq: error dynamics 1}. Therefore, due to \cite[Lemma~14.1, p.~589]{Khalil:2002}, the control given in \eqref{eq: stabilizer} guarantees that the trajectory of the system~\eqref{eq: error dynamics 1} is globally uniformly bounded. Moreover, the amplitude of regulation error $e(t)$ can be rendered arbitrarily small if the control gain $k_0$ is sufficiently high. This property is explicitly described in Theorem~\ref{theorem 12} in Section~\ref{Sec: Stability Analysis}.

{}{The schematic diagram of the considered aggregate power control system for TCL populations is shown in Fig.~\ref{Fig: scheme_control}, which has the same structure as most of the systems developed in the literature (see, e.g., \cite{Callaway:2009,Liu:2016,Liu2016MPC,Mahdavi:2017,tindemans2015decentralized,Totu:2017,Bashash2013,Ghaffari:2015,Ghanavati:2018}, etc.), except that MPC is used at the level of each TCL in a fully distributed manner. The computation of PDE control law requires only the value of the temperature of the TCLs and their current operational state. Whilst, the reference for MPCs is generated from the PDE control signal as $x_{\text{ref}}(t) = x_{\text{ref}}(0)+\int_{0}^{t}\dot{x}_{\text{ref}}(\tau)\text{d}\tau =  x_{\text{ref}}(0)+\int_{0}^{t}u(\tau)\text{d}\tau$. }

\begin{figure}[htpb]\
  \centering
  \includegraphics[scale=0.2]{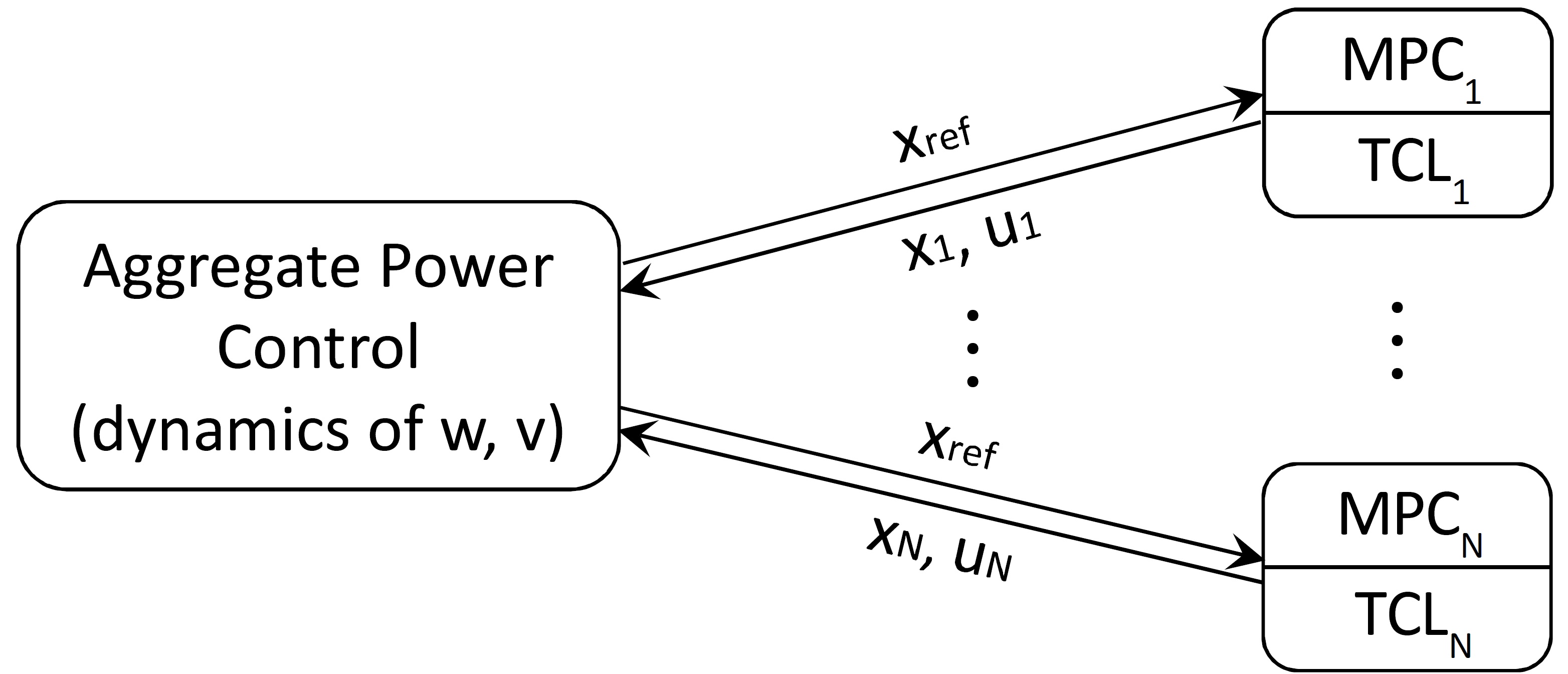}
  \caption{{}{Schematic diagram of aggregate power control of a TCL population operated by MPC at the TCL level.}
           }\label{Fig: scheme_control}
\end{figure}

\section{Well-posedness of the Closed-loop System}\label{Sec: Well-posedness}
In this section, we establish the well-posedness of the closed-loop system composed of \eqref{eq: foced FPE}, \eqref{eq: BCs of FPE}, \eqref{eq: output}, \eqref{eq: control u}, \eqref{eq: control Gamma}, \eqref{eq: error dynamics 1}, and \eqref{eq: stabilizer}. To this aim, we note that
$\alpha_1,\alpha_0 \in C^1([\underline{x},\overline{x}])$ and $P,\beta,\eta$ are positive constants, and assume that
\begin{align}
  &w^0,v^0\in \mathcal {H}^{2+\theta}([\underline{x},\overline{x}]),\delta \in L^\infty((0,\infty);L^\infty(\underline{x},\overline{x})),\label{22a}
\end{align}
with
\begin{subequations}
\begin{align}
  &\bigg|\int_{\underline{x}}^{\overline{x}}w^0(x)\text{d}x\bigg|>0, \label{eq: ON on 0}\\
  &\bigg|\int_{\underline{x}}^{\overline{x}}w^0(x)\text{d}x+\int_{0}^t\int_{\underline{x}}^{\overline{x}}\delta(x,s)\text{d}x\text{d}s\bigg|\geq \delta_0>0,\forall t\in [0,+\infty),\label{+28}
\end{align}
\end{subequations}
where $\theta\in (0,1),\delta_0>0$ are constants. Note that the conditions given in \eqref{eq: ON on 0} and \eqref{+28} mean that at any time the TCLs are not all in OFF-state.

{}{\begin{theorem}\label{well-posedness}{Consider the closed-loop system. For any $T>0$, there exists $(w,v)\in \mathcal {H}^{2+\theta,1+\frac{\theta}{2}}(\overline{Q}_{T})\times \mathcal {H}^{2+\theta,1+\frac{\theta}{2}}(\overline{Q}_{T})$ that satisfies \eqref{eq: foced FPE} a.e. in $ Q_T$.}
\end{theorem}}

{}{It should be pointed out that with the proposed control scheme, the closed-loop system is governed by parabolic PDEs with nonlinear nonlocal terms, which brings the following two obstacles for well-posedness analysis:
\begin{itemize}
   \item[(i)] Due to the appearance of a nonlinear nonlocal term in the form of ${w_x(x,t)\int \alpha_1(x)w(x,t)\text{d}x}/{\int w(x,t)\text{d}x}$, one cannot deal with the original PDEs directly and apply the Lyapunov method to obtain the \emph{a prior} estimates of the solutions. Consequently, the Galerkin method cannot be directly used for well-posedness analysis.
    \item[(ii)] The corresponding PDEs may have a singularity due to the appearance of the term $\int w(x,t)\text{d}x$ in the denominator of the nonlocal term. Therefore, the classical Leray-Schauder fixed-point theorem cannot be applied directly in well-posedness analysis.
\end{itemize}
To overcome the difficulties brought by the nonlinear nonlocal character, we use in the present work the technique of iteration for well-posedness analysis. }

{}{Specifically, for} any $T>0$, let $ \{\delta_n\}$ be a sequence of functions on $\overline{Q}_{T}$ satisfying:
\begin{itemize}
  \item $\delta_n(x,t)$ is H\"{o}lder continuous in $x$ with exponent $\theta$;
  \item $\delta_n(x,t)$ is $C^1$-continuous in $t$ with $\bigg |\frac{\partial\delta_{n}}{\partial t}(x,t)\bigg|\leq \widehat{\delta}$ in $\overline{Q}_T$, {}{where $ \widehat{\delta}$ is a positive constant};
  \item $\delta_n(x,t)\rightarrow \delta(x,t)$ a.e. in $Q_T$, as $n\rightarrow \infty$;
  \item $\|\delta_n\|_\infty\leq 1+\|\delta\|_\infty,\forall n\geq 1$;
  \item {}{for any $n\geq 1$ and $t\in [0,T]$, it holds: \begin{align}\bigg|\int_{\underline{x}}^{\overline{x}}w^0(x)\text{d}x+\int_{0}^t\int_{\underline{x}}^{\overline{x}}\delta_n(x,s)\text{d}x\text{d}s\bigg|\geq \frac{ \delta_0}{2}>0.\label{+28'}\end{align}}
\end{itemize}

Consider {}{the following two sets of iterating} equations in $Q_T$:
\begin{subequations}\label{+35}
\begin{align}
  &\dfrac{\partial w_n}{\partial t}(x,t) = -\dfrac{\partial}{\partial x}\left(\alpha_{1}(x)w_n(x,t)\right)
                                         +\beta \dfrac{\partial^2 w_n}{\partial x^2}(x,t)  -\dfrac{\partial w_n(x,t)}{\partial x}G_{w_{n-1}}(t)
   +\delta_n(x,t) ,  \label{+35a}\\
    &\beta \dfrac{\partial w_n}{\partial x}(\overline{x},t)-\alpha_1(\overline{x})w_{n}(\overline{x},t)-w_{n}(\overline{x},t)G_{w_{n-1}}(t)=0,\label{+35b}\\
    &\beta \dfrac{\partial w_n}{\partial x}(\underline{x},t)-\alpha_1(\underline{x})w_{n}(\underline{x},t)-w_{n}(\underline{x},t)G_{w_{n-1}}(t)=0,\label{+35c}\\
     &w_{n}(x,0)=w^0(x),\label{+35d}
\end{align}
\end{subequations}
and

\begin{subequations}\label{+v}
\begin{align}
 &\dfrac{\partial v_n}{\partial t}(x,t) = -\dfrac{\partial}{\partial x}\left(\alpha_{1}(x)v_n(x,t)\right)
                                         +\beta \dfrac{\partial^2 v_n}{\partial x^2}(x,t)  -\dfrac{\partial v_n(x,t)}{\partial x}G_{w_{n}}(t)
                                         +\delta_n(x,t) ,  \label{+35a'}\\
                                         &\beta \dfrac{\partial v_n}{\partial x}(\overline{x},t)-\alpha_0(\overline{x})v_{n}(\overline{x},t)-v_{n}(\overline{x},t)G_{w_{n}}(t)=0,\label{+35b'}\\
       &\beta \dfrac{\partial v_n}{\partial x}(\underline{x},t)-\alpha_0(\underline{x})v_{n}(\underline{x},t)-v_{n}(\underline{x},t)G_{w_{n}}(t)=0,\label{+35c'}\\
      & v_{n}(x,0)=v^0(x),\label{+35d'}
  \end{align}
\end{subequations}
where
\begin{align*}
G_{w_{{}{n-1}}}(t)=&\dfrac{\displaystyle\int_{\underline{x}}^{\overline{x}}\bigg(\beta \dfrac{\partial w_{{}{n-1}}}{\partial x}-\alpha_1 w_{{}{n-1}}\bigg)\text{d}x} {\displaystyle\int_{\underline{x}}^{\overline{x}}w_{{}{n-1}}\text{d}x}-\frac{k_0}{{}{a}}\dfrac{\displaystyle\int_{\underline{x}}^{\overline{x}}(ax+b)w_{{}{n-1}}\text{d}x -{}{\frac{\eta}{P}}y_d} {\displaystyle\int_{\underline{x}}^{\overline{x}}w_{{}{n-1}}\text{d}x},\ {}{n\geq 1},
\end{align*}
with {}{$w_0(x,t):= w^0(x)$}.

{}{We claim that for any $n$, there exists a unique solution $w_n,v_n\in \mathcal {H}^{2+\theta,1+\frac{\theta}{2}}(\overline{Q}_{T})$ to the iterating equations \eqref{+35} and \eqref{+v} respectively.
\begin{lemma}\label{well-posedness of iterating equaitons}
For any $n$, \eqref{+35} and \eqref{+v} admit a unique solution $w_n\in \mathcal {H}^{2+\theta,1+\frac{\theta}{2}}(\overline{Q}_{T})$ and $v_n\in \mathcal {H}^{2+\theta,1+\frac{\theta}{2}}(\overline{Q}_{T})$ respectively, satisfying {}{for any $t\in [0,T]$}:
\begin{align*}
&\int_{\underline{x}}^{\overline{x}}w_n(x,t)\text{d}x=\int_{\underline{x}}^{\overline{x}}w^0(x)
\text{d}x+\int_{0}^t\int_{\underline{x}}^{\overline{x}}\delta_n(x,s)\text{d}x\text{d}s\geq \frac{ \delta_0}{2},\\
&\int_{\underline{x}}^{\overline{x}}v_n(x,t)\text{d}x=\int_{\underline{x}}^{\overline{x}}v^0(x)
\text{d}x-\int_{0}^t\int_{\underline{x}}^{\overline{x}}\delta_n(x,s)\text{d}x\text{d}s\geq \frac{ \delta_0}{2},
\end{align*}
and
\begin{subequations}\label{eq: lemma 1 sequence}
\begin{align}
&\|w_n(\cdot,t)\|_1\leq  \|w^0\|_1
+\displaystyle\int_0^t\int_{\underline{x}}^{\overline{x}}\delta_n(x,s) \text{sgn} (w_n) \text{d}x\text{d}s, \label{eq: lemma 1 sequence w}\\
&\|v_n(\cdot,t)\|_1\leq  \|v^0\|_1
-\displaystyle\int_0^t\int_{\underline{x}}^{\overline{x}}\delta_n(x,s) \text{sgn} (v_n) \text{d}x\text{d}s  \label{eq: lemma 1 sequence v}.
\end{align}
\end{subequations}
Moreover, $(w_n,v_n)$ has the following uniform estimates:
\begin{subequations}\label{20191109}
\begin{align}
&\max_{\overline{Q}_{T}}|w_n|+\max_{\overline{Q}_{T}}|w_{nx}|+|w_n|^{(2+\theta)}_{Q_{T}}\leq C_1,\\
&\max_{\overline{Q}_{T}}|v_n|+\max_{\overline{Q}_{T}}|v_{nx}|+|v_n|^{(2+\theta)}_{Q_{T}}\leq C_2,
\end{align}
\end{subequations}
where $C_1$ and $C_2$ are positive constants independent of $n$.
\end{lemma}}

\begin{proof}
The proof is proceeded in four steps as detailed in {}{Appendix~\ref{Appendix: Lemma 1}}.
\end{proof}
\begin{remark}  {}{The main mathematical framework used in the proof of Lemma \ref{well-posedness of iterating equaitons} is as follows:}
{}{\begin{itemize}
\item [(i)] The existence of a unique solution $(w_n,v_n)\in \mathcal {H}^{2+\theta,1+\frac{\theta}{2}}(\overline{Q}_{T})\times \mathcal {H}^{2+\theta,1+\frac{\theta}{2}}(\overline{Q}_{T})$ to the iterating equations \eqref{+35} and \eqref{+v} is proved by induction combining with the well-posedness and regularity theory for quasilinear parabolic equations in \cite[Chap. V]{Ladyzenskaja:1968}.
   \item [(ii)] The $L^1$-estimates of $w_n$ and $v_n$ in \eqref{eq: lemma 1 sequence} are established by utilizing appropriate test functions and some basic integrating techniques. Note that since $\|\delta_n\text{sgn} (w_n)\|_\infty\leq 1+\|\delta\|_\infty$ and $\|\delta_n\text{sgn} (v_n)\|_\infty\leq 1+\|\delta\|_\infty$, \eqref{eq: lemma 1 sequence} implies the uniform boundedness of $\|w_n\|_1$ and $\|v_n\|_1$, respectively.
  \item[(iii)] The uniform boundedness of $\|w_n\|_1$ and $\|v_n\|_1$ and the continuity of $w_n$ and $v_n$ guarantee the uniform boundedness of $w_n$ and $v_n$, which, along with the well-posedness and regularity theory for quasilinear parabolic equations in \cite[Chap. V]{Ladyzenskaja:1968}, guarantees the uniform estimates of $w_n$ and $v_n$ in \eqref{20191109}.
    \end{itemize}}
\end{remark}
Based on Lemma~\ref{well-posedness of iterating equaitons}, we can establish the existence and the regularity of a solution $(w,v)$ to the coupled equations \eqref{eq: foced FPE} in closed loop formed by \eqref{eq: foced FPE}, \eqref{eq: BCs of FPE}, \eqref{eq: output}, \eqref{eq: control u}, \eqref{eq: control Gamma}, \eqref{eq: error dynamics 1}, and \eqref{eq: stabilizer}.
\begin{proof}[Proof of Theorem~\ref{well-posedness}]
By Lemma~\ref{well-posedness of iterating equaitons}, {}{$(w_n,v_n)$ is bounded in $\mathcal {H}^{2+\theta,1+\frac{\theta}{2}}(\overline{Q}_{T})\times \mathcal {H}^{2+\theta,1+\frac{\theta}{2}}(\overline{Q}_{T})$. Then, by choosing a subsequence, there is a pair of functions $(w,v)\in \mathcal {H}^{2+\theta,1+\frac{\theta}{2}}(\overline{Q}_{T})\times \mathcal {H}^{2+\theta,1+\frac{\theta}{2}}(\overline{Q}_{T})$ satisfying $(w_n,v_n)\rightarrow (w,v)$ in $ \mathcal {H}^{2+\theta,1+\frac{\theta}{2}}(\overline{Q}_{T})\times \mathcal {H}^{2+\theta,1+\frac{\theta}{2}}(\overline{Q}_{T})$ as $n\rightarrow \infty$. Noting that $\delta_n\rightarrow \delta$ a.e. in $ Q_T$, we conclude that $(w,v)$ satisfies \eqref{eq: foced FPE} a.e. in $ Q_T$.}
\end{proof}

It should be noticed that Theorem \ref{well-posedness} can still not guarantee the uniqueness of a solution of the closed-loop system, because the solution is obtained by the technique of iteration and approximation. Moreover, in general it is not trivial to establish the uniqueness of a solution to an equation with strong nonlinearities. {}{Nevertheless, the uniqueness of the solution of the closed-loop system can be guaranteed in a weak sense as stated below.
\begin{theorem}\label{Proposition 9}
Consider the closed-loop system. Let $(w_1,v_1),(w_2,v_2)\in \mathcal {H}^{2+\theta,1+\frac{\theta}{2}}(\overline{Q}_{T})\times \mathcal {H}^{2+\theta,1+\frac{\theta}{2}}(\overline{Q}_{T})$ be two solutions of \eqref{eq: foced FPE} in $ Q_T$.
If $w_1(\overline{x},t)-w_1(\underline{x},t)=w_2(\overline{x},t)-w_2(\underline{x},t)$ in $[0,T]$, then $(w_1,v_1)=(w_2,v_2)$ in $\overline{Q}_T$.
\end{theorem}}
\begin{proof}
See {}{Appendix~\ref{App: uniqueness}}.
\end{proof}

\section{Stability Analysis}\label{Sec: Stability Analysis}
{}{We assess first the stability of the error dynamics~\eqref{eq: error dynamics 1} with the control given in \eqref{eq: stabilizer}.}
\begin{theorem}\label{theorem 12}
The regulation error $e(t)$ is determined by
\begin{equation}
  e(t)=e(0)\text{e}^{-k_0t}+\int_{0}^t\Gamma(t)\text{e}^{-k_0(t-s)}\text{d}s.\label{1125+4}
\end{equation}
Furthermore, if $|\Gamma(t)| \leq \Gamma_{\infty}$ with a positive constant $\Gamma_{\infty}$ for all $t > 0$, then
\begin{equation*}
|e(t)|\leq |e(0)|\text{e}^{-k_0t}+\frac{\Gamma_{\infty}}{k_0}\left(1-\text{e}^{-k_0t}\right),\forall t>0.
\end{equation*}
\end{theorem}

{\begin{proof} By \eqref{eq: error dynamics 1} and \eqref{eq: stabilizer}, it follows
\begin{equation*}
  \dfrac{\diff{e}}{\diff{t}} = -k_0e(t) + \Gamma(t), \ \ e(0) = e_0.
\end{equation*}
Solving this first-order linear ordinary differential equation, we can obtain the result given in~\eqref{1125+4}.
\end{proof}}

{}{To assure the closed-loop stability, we need to prove that the internal dynamics composed of \eqref{eq: foced FPE}, \eqref{eq: BCs of FPE}, \eqref{eq: output}, and \eqref{eq: control u} are stable provided Theorem~\ref{theorem 12} holds true. Towards this aim, we need the properties stated in the following 2~propositions. The first one is the conservativity of mass of the considered problem given below.
}

\begin{proposition}\label{prop. 8}
Consider the closed-loop system. For any $T>0$, let $(w,v)\in \mathcal {H}^{2+\theta,1+\frac{\theta}{2}}(\overline{Q}_{T})\times \mathcal {H}^{2+\theta,1+\frac{\theta}{2}}(\overline{Q}_{T})$ satisfy \eqref{eq: foced FPE} a.e. in $ Q_T$. Then for any $t\in [0,T]$, it holds:
\begin{enumerate}
\item [(i)]
{}{$\displaystyle\int_{\underline{x}}^{\overline{x}}w(x,t)\text{d}x=\displaystyle\int_{\underline{x}}^{\overline{x}}w^0(x)\text{d}x+\int_{0}^t\int_{\underline{x}}^{\overline{x}}\delta(x,s)\text{d}x\text{d}s;$}
\item [(ii)]
{}{$\displaystyle\int_{\underline{x}}^{\overline{x}}v(x,t)\text{d}x=\displaystyle\int_{\underline{x}}^{\overline{x}}v^0(x)\text{d}x-\int_{0}^t\int_{\underline{x}}^{\overline{x}}\delta(x,s)\text{d}x\text{d}s.
$}
\end{enumerate}
Moreover, it holds:
\begin{align*}
&N_{\text{agg}}(t) =\int_{\underline{x}}^{\overline{x}}(w(x,t)+v(x,t))\text{d}x=\int_{\underline{x}}^{\overline{x}} (w^0(x)+ v^0(x))\text{d}x.
\end{align*}
\end{proposition}
\begin{proof}
The result can be obtained directly by taking integration of \eqref{eq: foced FPE} and using the boundary conditions \eqref{eq: BCs of FPE}.
\end{proof}


{}{In order to establish the stability of the closed-loop system, we need the following property on the solutions of \eqref{eq: foced FPE}, \eqref{eq: BCs of FPE}, \eqref{eq: output}, \eqref{eq: control u}, \eqref{eq: control Gamma}, \eqref{eq: error dynamics 1}, and \eqref{eq: stabilizer}.}

\begin{proposition}\label{prop. 7}
Consider the closed-loop system. For any $T>0$, let $(w,v)\in \mathcal {H}^{2+\theta,1+\frac{\theta}{2}}(\overline{Q}_{T})\times \mathcal {H}^{2+\theta,1+\frac{\theta}{2}}(\overline{Q}_{T})$ be a solution of \eqref{eq: foced FPE} in $ Q_T$. Then for any $t\in [0,T]$, the following estimates hold
\begin{enumerate}
\item [(i)]
{}{$
\|w(\cdot,t)\|_1\leq  \|w^0\|_1
+\displaystyle\int_0^t\int_{\underline{x}}^{\overline{x}}\delta(x,s) \text{sgn} (w) \text{d}x\text{d}s;$}
\item [(ii)]{}{$\|v(\cdot,t)\|_1\leq  \|v^0\|_1
-\displaystyle\int_0^t\int_{\underline{x}}^{\overline{x}}\delta(x,s) \text{sgn} (v) \text{d}x\text{d}s,$}
\end{enumerate}
{}{where $\text{sgn} (\eta)=1$ if $\eta>0$, $\text{sgn} (\eta)=-1$ if $\eta<0$, and $\text{sgn} (\eta)=0$ if $\eta=0$.}
\end{proposition}
{}{\begin{proof}
By Lemma~\ref{well-posedness of iterating equaitons} and the regularity of $w_n$ and $v_n$, it suffices to let $n \rightarrow \infty$ in \eqref{eq: lemma 1 sequence}, which leads to the desired result.
\end{proof}}

As it is guaranteed by Proposition~\ref{prop. 8} that the closed-loop system is also conservative in terms of the total number of TCLs in the population, {}{there must be a positive constant $M$ such that $\left|\iint_Q\delta(x,{t}) \text{d}x\text{d}t\right|<M$ for any Lebesgue measurable set $Q\subset Q_{\infty}$.} We have then the following result regarding the stability of the internal dynamics in closed loop.
\begin{theorem}\label{stability}
Consider the internal dynamics of the system. Let $(w,v)\in \mathcal {H}^{2+\theta,1+\frac{\theta}{2}}(\overline{Q}_{\infty})\times \mathcal {H}^{2+\theta,1+\frac{\theta}{2}}(\overline{Q}_{\infty})$ satisfy \eqref{eq: foced FPE} a.e. in $ \overline{Q}_\infty$. Assume that {}{there exists a positive constant $M$ such that} {}{$\left|\iint_Q\delta(x,{}{t}) \text{d}x\text{d}t\right|<M$ for any Lebesgue measurable set $Q\subset Q_{\infty}$.} Then for any $t\in [0,+\infty)$, it holds:
\begin{enumerate}
\item [(i)]
{}{$
\|w(\cdot,t)\|_1\leq  \|w^0\|_1
+{}{2M}<+\infty;$}
\item [(ii)]{}{$\|v(\cdot,t)\|_1\leq  \|v^0\|_1+
{}{2M}<+\infty.$}
\end{enumerate}
\end{theorem}
\begin{proof}
We asses only the claim~(i). The proof of the claim~(ii) can follow the same line. Indeed, {}{for $t=0$, we have yet $\|w(\cdot,0)\|_1= \|w^0\|_1\leq \|w^0\|_1
+2{}{M}<+\infty$. For any $t>0$, let $Q^+=\{(x,s)\in \overline{Q}_{{}{t}};w(x,s)>0\} $ and $Q^-=\{(x,s)\in \overline{Q}_{{}{t}};w(x,s)<0\} $.}
 %
{}{Then by assumptions and Proposition \ref{prop. 7}, we have
\begin{align*}
\|w(\cdot,t)\|_1\leq&  \|w^0\|_1
+\displaystyle\int_0^t\int_{\underline{x}}^{\overline{x}}\delta(x,s) \text{sgn} (w) \text{d}x\text{d}s\\
=&\|w^0\|_1+\iint_{Q^+}\delta(x,s)  \text{d}x\text{d}s-\iint_{Q^-}\delta(x,s)  \text{d}x\text{d}s\\
\leq &\|w^0\|_1+\left|\iint_{Q^+}\delta(x,s)  \text{d}x\text{d}s\right|+\left|\iint_{Q^-}\delta(x,s)\text{d}x\text{d}s\right|\\
\leq & \|w^0\|_1+2M<+\infty.
\end{align*}}
\end{proof}

\section{Simulation Study}\label{Sec: Simulation Study}
To validate the developed control scheme and illustrate the performance of the system, we conduct a simulation study on a benchmark problem, corresponding to a population of air-conditioners, which has been used in the validation of different solutions in the literature, e.g., {\cite{Callaway:2009,moura2013modeling,Ghanavati:2018,Moura:2014}}. The parameters of the system and the controllers are listed in Table~\ref{Tab: system parameters}. The variation of thermal capacitance creates a heterogenous TCL population.
\begin{table}[thpb]\
\caption{System parameters}\label{Tab: system parameters}
\centering
\begin{tabular}{|l|c|c|c|c|c|}
\hline
Parameters & Description [Unit] & Value\\
\hline
$R$ & Thermal resistance [$^\circ$C/kW] & 2\\
$C$ & Thermal capacitance [kWh/$^\circ$C] & $ N(10,3)$\\
$P$ & Thermal power [kW] & 14\\
$x_{ie}$ & Ambient temperature [$^\circ$C] & 32\\
$\eta$ & Coefficient of performance & 2.5 \\
$\beta$ & Diffusivity & 0.1 \\
\hline
$Q_{\text{mpc}}$ & Weighting coefficient on tracking error $e_{\text{mpc}}$ & 100\\
$R_{\text{mpc}}$ & Weighting coefficient on control input $u_{\text{mpc}}$ & $ N(10,2)$\\
$M$ & Prediction horizon & 5 \\
$\underline{\Delta x}$, $\overline{\Delta x}$ & Temperature deadband width [$^\circ$C] & 0.5\\
$a$ & Weighting function coefficient & -1 \\
$b$ & Weighting function coefficient & -20 \\
$k_0$ & Control gain & 5600\\
\hline
\end{tabular}
\end{table}

The setting used in the simulation is a population of 1000 appliances controlled by MPC schemes. The simulation is performed in a normalized time scale with respect to the nominal time constant of the appliances represented by $RC$ (=~20h). The results related to power consumptions are also presented in quantities normalized by the maximal total power consumption of the population. {}{Note that taking $b = -20$ in the weighting function used in the output function defined in \eqref{eq: output} has the effect of moving the data around the center of the operation range. Our experiment has shown that this choice would facilitate the controller tuning.}

First, it is known that due to the phenomenon of \emph{cold load pickup} \cite{Ihara1981}, a step variation of temperature reference in a large population of TCLs may induce high power demand peaks. To avoid this problem, we have adopted a decentralized desynchronization mechanism that combines adding independently a weighting coefficient and a delay to the reference signal at the level of MPC for individual TCLs \cite{laparra2019}. To illustrate the effect of this desynchronized MPC scheme, we tested the response of the population to a step change of the reference from 20.5~$^\circ$C to 19~$^\circ$C occurred at $t = 15$. The MPCs are with a weighting coefficient $R_{\text{mpc}}$ from a normally distributed random variable restricted in the interval [6, 14] and an additional delay uniformly distributed in [0, 5] to the reference signal. It can be seen from Fig.~\ref{Fig: Aggregate PW} that the desynchronized MPC scheme can considerably reduce the peak power demand.

\begin{figure}[htpb]\
  \centering
  \includegraphics[scale=0.2]{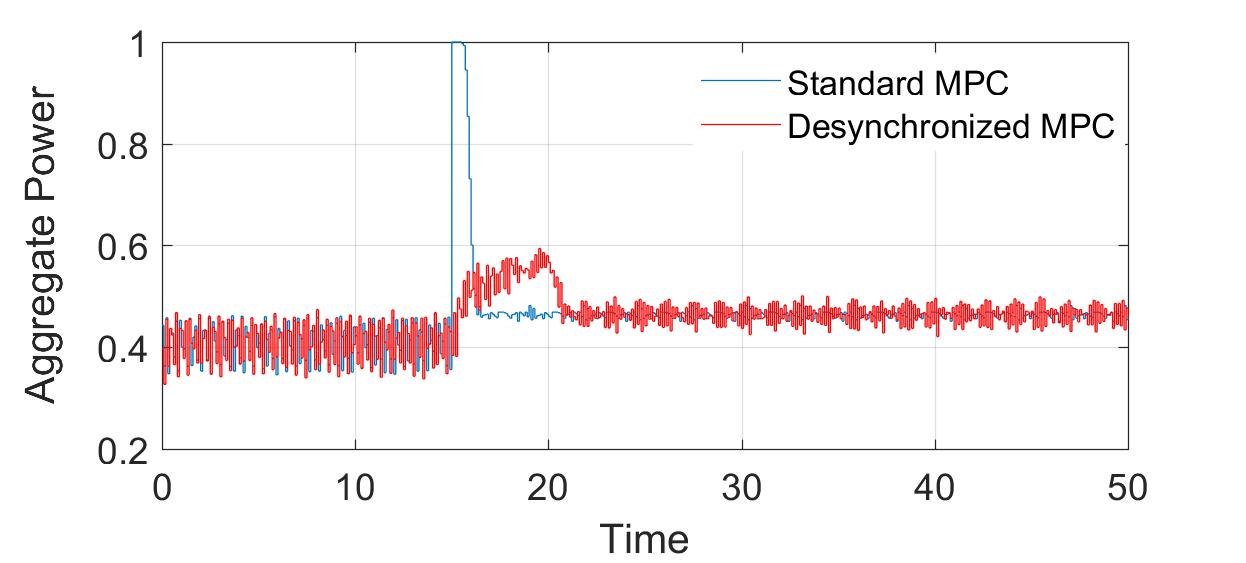}
  \caption{{}{Comparison of aggregate power of the population controlled by standard MPC and desynchronized MPC schemes at the TCL level.}
           }\label{Fig: Aggregate PW}
\end{figure}

The second part of simulation study is to determine the diffusion coefficient, $\beta$, of the Fokker-Planck equations based on the experiments. The diffusivity is estimated by applying the algorithm developed in \cite{Moura:2014} to the data generated by numerical simulations, which gives a value $\beta=0.1$. To validate this result, we proceed by first generating the data $\delta(x,t)$ for the population of TCLs specified in Table~\ref{Tab: system parameters} with MPC at the TCL level. The numerical solutions of the Fokker-Planck equations with the generated data $\delta(x,t)$ as the inputs corresponding to different values of $\beta$ are obtained by the forward Euler method \cite{hildebrand1987introduction}. The solution surfaces of the Fokker-Planck equations, $w(x,t)$ and $v(x,t)$, with $\beta = 0.1$ are shown in Fig.~\ref{Fig: Distribution}. We compute then the normalized mean error in its absolute value of the number of the TCLs in ON and OFF states between the generated data and that given by the solution of the Fokker-Planck equations with different diffusion coefficient varying from 0 to 0.2. As can be seen from Fig.~\ref{Fig: beta} that $\beta = 0.1$ is an adequate estimate. It can also be observed that the errors are not sensitive to diffusion coefficient if it is smaller than 0.1. This is in accordance with the results reported in the literature (see, e.g., \cite{Moura:2014}).

\begin{figure}[htpb]\
  \centering
  \subfigure[]{\label{Fig: Distribution}\includegraphics[scale=0.2]{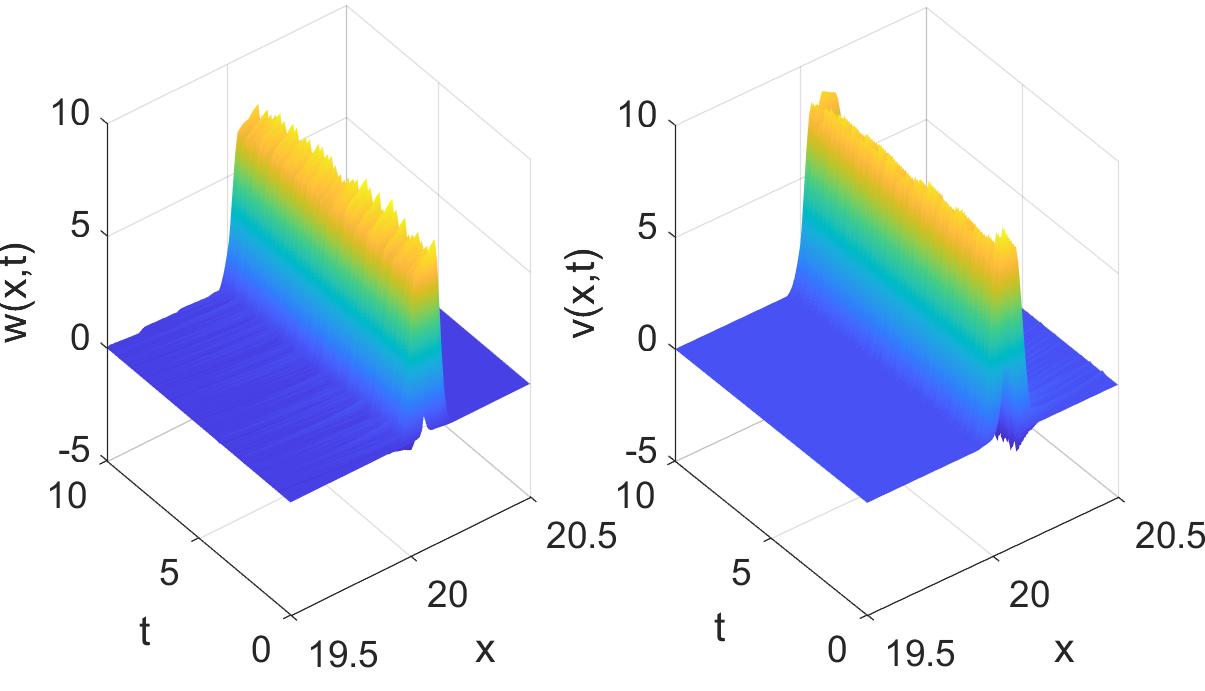}}
  \subfigure[]{\label{Fig: beta}\includegraphics[scale=0.2]{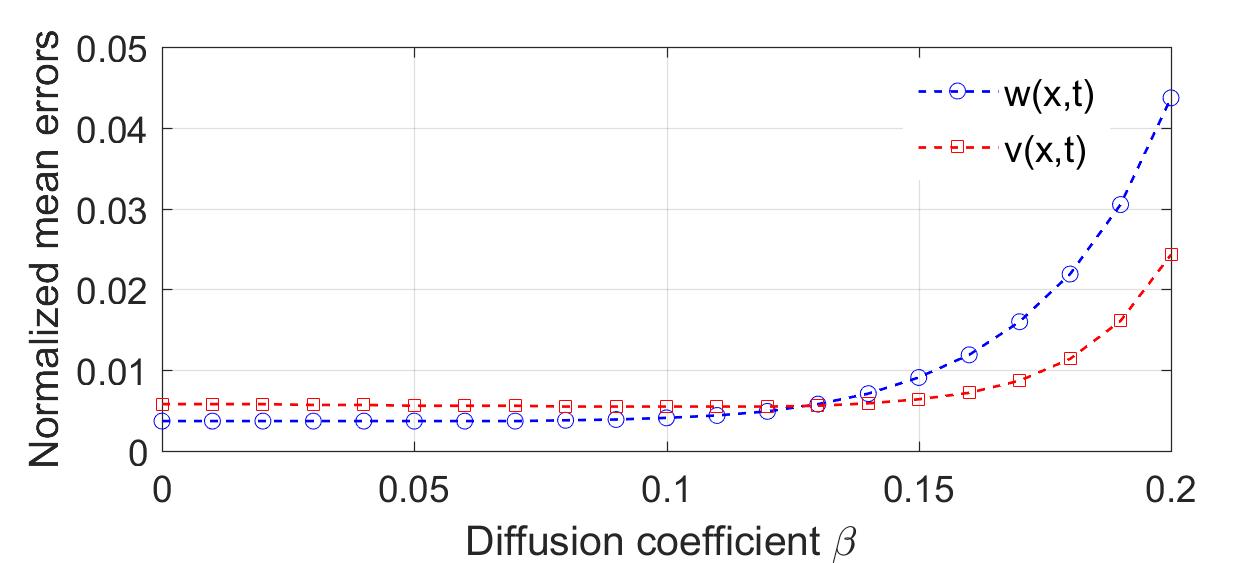}}
  \caption{{}{Experimental validation of the estimate of diffusion coefficient:
           (a) solution surfaces of the distribution of TCLs in ON and OFF states: $w(x,t)$ and $v(x,t)$;
           (b) normalized mean errors of the number of TCLs in ON and OFF states between the measured value and the one obtained by the solution of Fokker-Planck equations.
           }}
\end{figure}

{}{In the last part of simulation study, we test the developed aggregate power control algorithm with} a case where the desired aggregate power changes from 0.41 to 0.47 at $t = 15$ (all in the normalized scale), which corresponds to a variation of the reference temperature from 20.5$^\circ$C to 19$^\circ$C. {\cite{Callaway:2009,Bashash2013,moura2013modeling}}. The desired aggregate power and the actual power consumption of the population are shown in Fig.~\ref{Fig: Power}. The control signal and the regulation errors are depicted in Fig.~\ref{Fig: Command} and Fig.~\ref{Fig: Error}, respectively. The results show that the developed control scheme performances well with regulation errors lower than 2.5\% of the maximal total power consumption at the steady regime.
In order to avoid the chartering induced by fast reference variations, the reference sent to the TCLs is the average of the values over the last 10 iterations. The obtained temperature reference is shown in Fig.~\ref{Fig: Reference}. Figure~\ref{Fig: Temperature100} illustrates the temperature evolution of 100 randomly selected TCLs. It can be seen that all the TCLs follow well the reference while respecting the temperature deadband. {}{Finally, it is worth noting that in the simulation we used a population containing 1000 TCLs. As the bigger is the size of the population, the more accurate is the continuum model, along with the fact that the switching of the TCLs is considered unknown as specified in \eqref{eq: error dynamics 1}, the simulation results conformed the robustness of the developed control scheme.}

\begin{figure}[htpb]\
  \centering
  \subfigure[]{\label{Fig: Power}\includegraphics[scale=0.2]{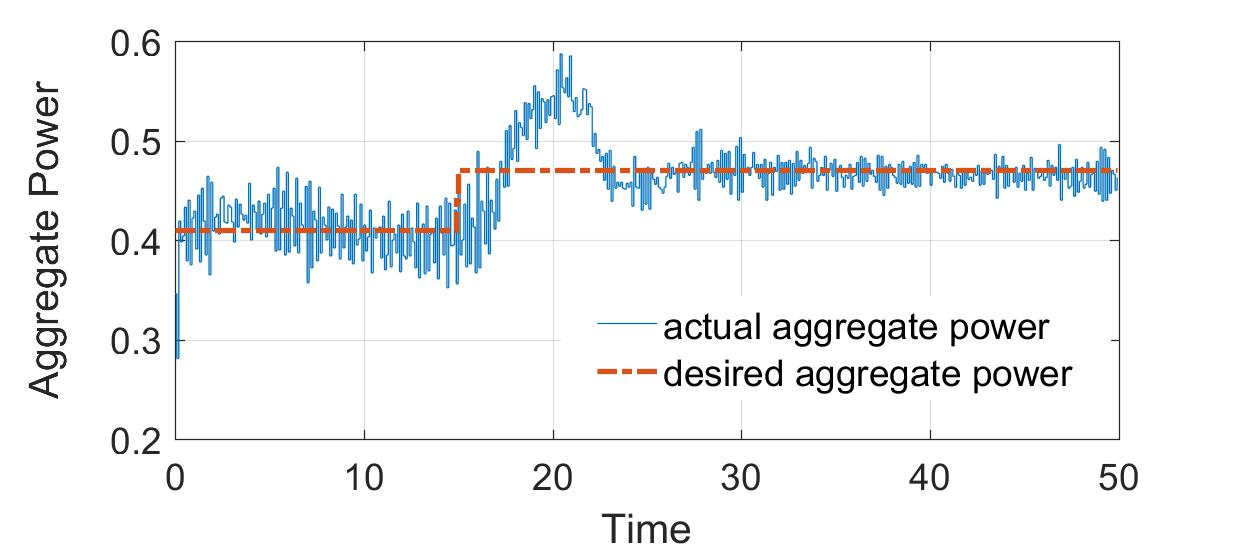}}
  \subfigure[]{\label{Fig: Command}\includegraphics[scale=0.2]{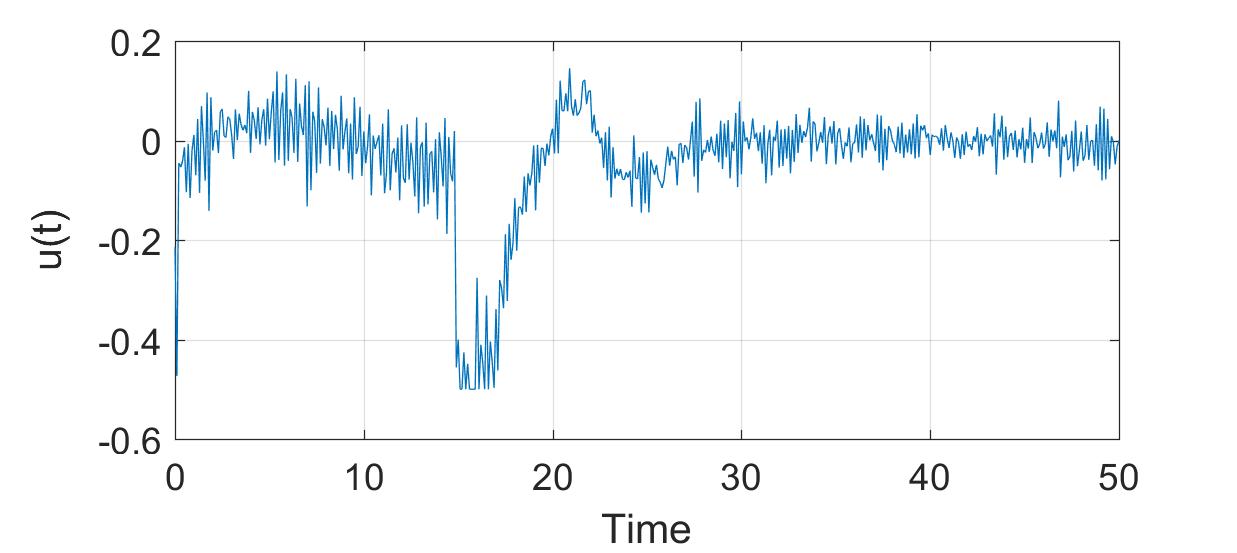}}
  \subfigure[]{\label{Fig: Error}\includegraphics[scale=0.2]{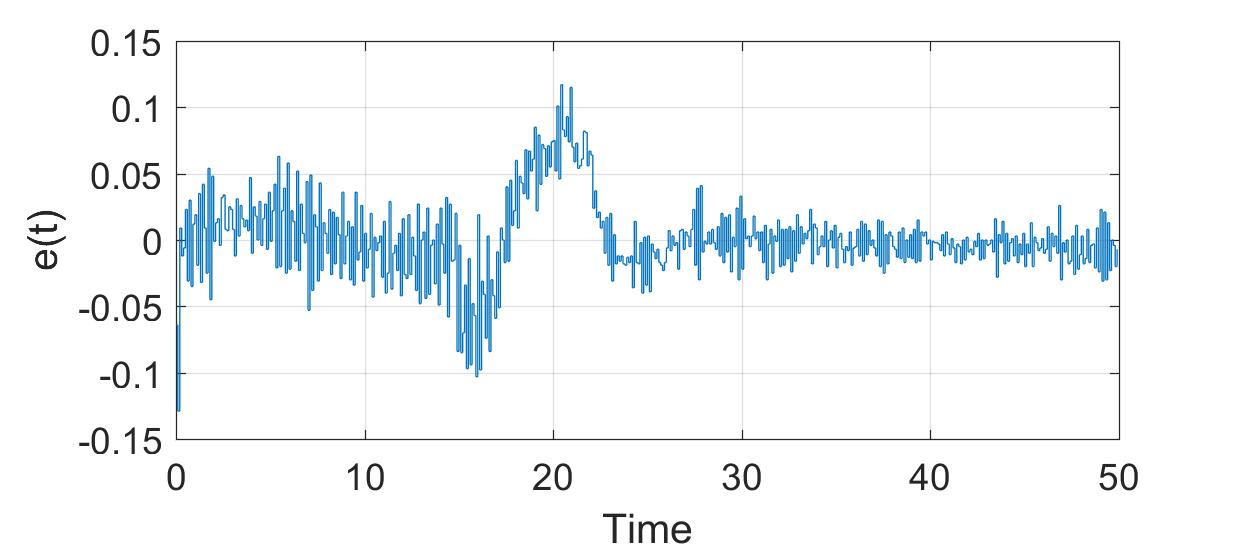}}
  \caption{Dynamics of the TCL population:
           (a) aggregate power consumptions;
           (b) control signal;
           (c) regulation error.
           }
\end{figure}

\begin{figure}[t]\
  \centering
  \subfigure[]{\label{Fig: Reference}\includegraphics[scale=0.2]{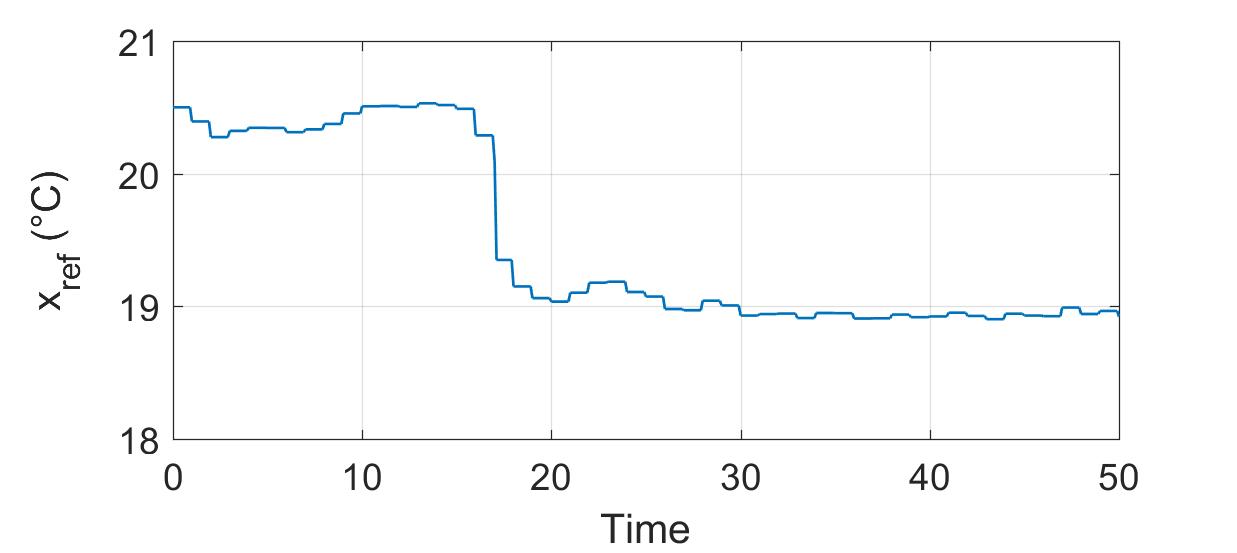}}
  \subfigure[]{\label{Fig: Temperature100}\includegraphics[scale=0.2]{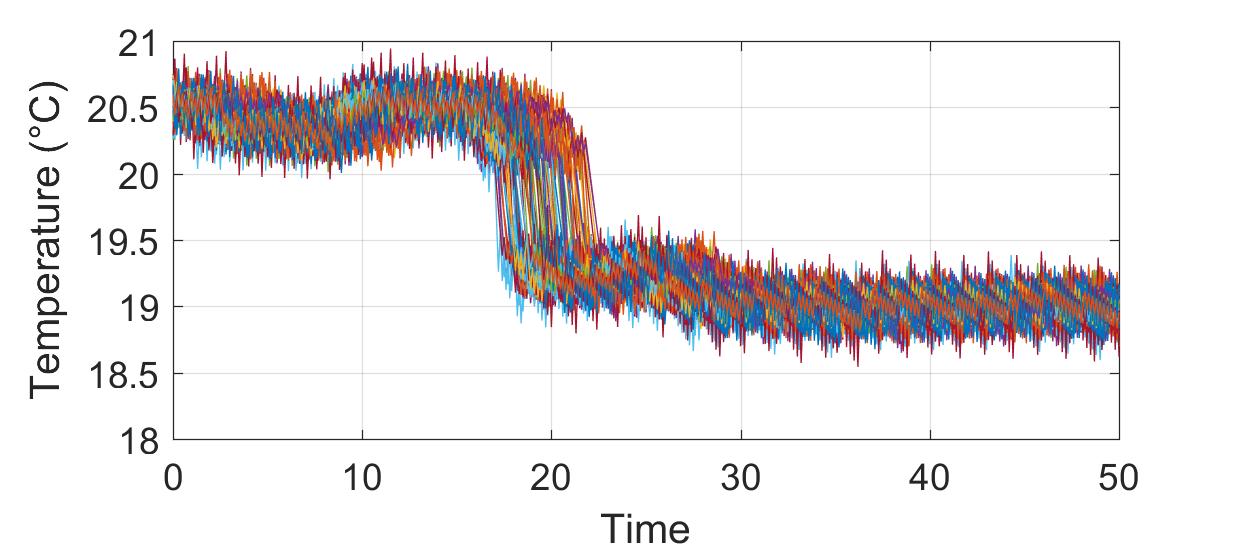}}
  \caption{Temperature evolution:
           (a) reference temperature;
           (b) temperature variations of 100 samples of TCLs.
           }
\end{figure}


\section{Conclusion}\label{Sec: Conclusion}
In this work, we have developed an input-output linearization-based scheme for aggregate power control of heterogenous thermal thermostatically controlled load (TCL) populations governed by a pair of Fokker-Planck equations coupled via in-domain actions representing the operation of model predictive control (MPC) at the level of individual TCLs. Well-posedness and stability analysis on the closed-loop system is conducted. A simulation study is carried out with a benchmark setting, and the results illustrated the validity and the performance of the proposed method.

It should be noted that the present work considered only a set-point regulation problem. While as the proposed linearization control leads to a system with finite-dimensional linear input-output dynamics, it is straightforward to extend this scheme to aggregate power tracking by using the well-established techniques for tracking control of finite-dimensional systems {\cite{Khalil:2002,Krstic:1995,Isidori:1995,Levine:2009}}. Finally, the control algorithm developed in this work can be applied to TCL populations operated by other control schemes at the level of individual TCLs, in particular the variations of deadband control. The main difference will be the way with which the PDEs are coupled, which will lead to settings with different boundary conditions and in-domain inputs. Hence, the well-posedness and the stability of the closed-loop systems have to be assessed, which can also be conducted following the approach used in this work.

{{}
\section*{Acknowledgements}
The authors thank the anonymous reviewers for their constructive comments, which helped to improve the quality of this work and the presentation of the paper.
}

{{}
\section{Appendix}
\subsection{Notations on $\mathcal {H}^{l}([\underline{x},\overline{x}])$ and $\mathcal {H}^{l,l/2}(\overline{Q}_T)$}\label{Appendix: notation}
}
We introduce first the following definitions of function spaces given in \cite[Chapter~I]{Ladyzenskaja:1968}, which are used in the subsequent development. Let $T>0$ and $l$ be a nonintegral positive number. {}{$\mathcal {H}^{l}([\underline{x},\overline{x}])$} denotes the Banach space whose elements are continuous functions $u(x)$ in $[\underline{x},\overline{x}]$, having continuous derivatives up to order $[l]$ inclusively and a finite value for the quantity
\begin{align}
    &|u|^{(l)}_{(\underline{x},\overline{x})}=\langle u\rangle^{(l)}_{(\underline{x},\overline{x})}+\sum_{j=0}^{[l]}\langle u\rangle_{(\underline{x},\overline{x})}^{(j)},\label{1.9}\\
    &\langle u\rangle^{(0)}_{(\underline{x},\overline{x})}=
    |u|^{(0)}_{(\underline{x},\overline{x})}=\sup_{(\underline{x},\overline{x})}|u|, \notag\\
    &\langle u\rangle_{(\underline{x},\overline{x})}^{(j)}= \sum_{(j)}|D^j_xu|^{(0)}_{(\underline{x},\overline{x})}, \langle u\rangle_{(\underline{x},\overline{x})}^{(l)}= \sum_{([l])}|D^{[l]}_xu|^{(l-[l])}_{(\underline{x},\overline{x})},\notag
\end{align}
where \eqref{1.9} defines the norm $|u|^{(l)}_{(\underline{x},\overline{x})} $ in $ \mathcal {H}^{l}([\underline{x},\overline{x}])$.

$\mathcal {H}^{l,l/2}(\overline{Q}_T)$ denotes the Banach space of functions $u(x,t)$ that are continuous in $\overline{Q}_T$, together with all derivatives of the form $D^r_tD^s_x$ for $2r+s<l$, and have a finite norm
\begin{align}
    &|u|^{(l)}_{Q_T}= \langle u\rangle^{(l)}_{Q_T}+\sum_{j=0}^{[l]}\langle u\rangle^{(j)}_{Q_T},\label{1.10}
\end{align}
with
\begin{align}
    \langle u\rangle^{(0)}_{Q_T}=&|u|^{(0)}_{Q_T}=\max_{\overline{Q}_T}|u|, \notag\\
    \langle u\rangle^{(j)}_{Q_T}=& \sum\limits _{(2r+s=j)}|D^r_tD^s_xu|^{(0)}_{Q_T},\notag\\
    \langle u\rangle^{(l)}_{Q_T}=&\langle u\rangle^{(l)}_{x,Q_T}+\langle u\rangle^{(l/2)}_{t,Q_T}, \notag\\
    \langle u\rangle^{(l)}_{x,Q_T}=& \sum\limits _{(2r+s=[l])}\langle D^r_tD^s_xu\rangle^{(l-[l])}_{x,Q_T},   \notag\\
    \langle u\rangle^{(l/2)}_{t,Q_T}= & \sum\limits _{0<l-2r-s<2}\langle D^r_tD^s_xu\rangle^{(\frac{l-2r-s}{2})}_{t,Q_T}, \notag\\
    \langle u\rangle^{(\alpha)}_{t,Q_T}=&\sup_{(x,t),(x,t')\in \overline{Q}_T,|t-t'|\leq \rho_0}\frac{|u(x,t)-u(x,t')|}{|t-t'|^{\tau}}, \notag\\
     & \ 0<\tau <1,\notag\\
     \langle u\rangle^{(\tau)}_{x,Q_T}=&\sup_{t\in (0,T)}\sup \rho^{-\tau}\text{osc}_x\left\{u(x,t);\Omega_\rho^i\right\}.\notag
\end{align}
In the last line, the second supremum is taken over all connected components $ \Omega_\rho^i$ of all $\Omega_\rho$ with $\rho{}{\leq} \rho_0$, while {}{$\text{osc}_x\left\{u(x,t);\Omega_\rho^i\right\}$ is the oscillation of $u(x,t)$ on $\Omega_\rho^i$, i.e.,} $\text{osc}_x\left\{u(x,t);\Omega_\rho^i\right\} =\left\{\text{vrai}\sup\limits_{x\in \Omega_\rho^i}u(x,t)-\text{vrai}\inf\limits_{x\in \Omega_\rho^i}u(x,t)\right\}$, {}{$\rho_0=\underline{x}-\overline{x}$}, $\Omega_\rho=K_\rho\cap (\underline{x},\overline{x}) $, and $K_\rho$ is an arbitrary open interval in $(-\infty,+\infty)$ of length $2\rho$.

{{}
\subsection{Proof of Lemma~\ref{well-posedness of iterating equaitons}}\label{Appendix: Lemma 1}
}
The proof of Lemma~\ref{well-posedness of iterating equaitons} can be proceeded in four steps.

\textbf{Step 1:} We prove the existence of solutions of  \eqref{+35} and \eqref{+v} {}{by induction combining with the generalization of the approach developed in \cite{Ladyzenskaja:1968}}.
\begin{lemma}\label{Proposition: existence of approximation}
For any $n$, \eqref{+35} and \eqref{+v} have a unique solution $w_n\in \mathcal {H}^{2+\theta,1+\frac{\theta}{2}}(\overline{Q}_{T})$ and $v_n\in \mathcal {H}^{2+\theta,1+\frac{\theta}{2}}(\overline{Q}_{T})$ respectively.
\end{lemma}
\begin{proof}
{First} we prove the existence and the uniqueness of the solution $w_n\in \mathcal {H}^{2+\theta,1+\frac{\theta}{2}}(\overline{Q}_{T})$ of \eqref{+35}. {}{We rewrite \eqref{+35} as below:
\begin{subequations}\label{+35'}
\begin{align}
  &\dfrac{\partial w_n}{\partial t}(x,t)-\beta w_{nxx} + b_n(x,t,w_n,w_{nx})=0,  \label{+35a'}\\
  &\beta w_{nx}(\overline{x},t)+\psi_n(\overline{x},t,w_n)=0,\label{+35b'}\\
  &\beta w_{nx}(\underline{x},t)-\psi_n(\underline{x},t,w_n)=0,\label{+35c'}\\
                                         &w_{n}(x,0)=w^0(x),\label{+35d'}
\end{align}
\end{subequations}
where
$b_n(x,t,w,p)=\alpha_{1x}w+(\alpha_1+G_{w_{n-1}}(t))p-\delta_n(x,t)$ and $ \psi_n(x,t,w)=(\underline{x}+\overline{x}-2x)(\alpha_1+G_{w_{n-1}}(t))w$.}

We proceed by induction {}{to show that \eqref{+35'} has a unique solution $w_n\in \mathcal {H}^{2+\theta,1+\frac{\theta}{2}}(\overline{Q}_{T})$ which satisfies:
\begin{align}\label{+34'''}
\max_{\overline{Q}_{T}}|w_n|\leq m_n,
\max_{\overline{Q}_{T}}|w_{nx}|\leq M_{n},|w_n|^{(2+\theta)}_{Q_{T}}\leq \overline{M}_n,
\end{align}
where $m_n$ depends only on $a,b,k_0,y_d,P,\eta,\beta$, $\|\alpha_1\|_\infty$, $\|\alpha_{1x}\|_\infty$, $\|\delta\|_\infty$, $\max_{\overline{Q}_{T}}|w_{n-1}|,\big|\int_{\underline{x}}^{\overline{x}}w^{0}\text{d}x\big|$ and $n$, $ M_{n}$ depends only on $\beta$, $\overline{x}$, $\underline{x}$, $m_n$, $\|\alpha_1\|_\infty$, $\|\alpha_{1x}\|_\infty$, $\max_{\overline{Q}_{T}}|w_{n-1}|$, $|w^0|^{(2)}_{(\underline{x},\overline{x})}$, and $\overline{M}_n$ depends only on $M_n$, $m_n$, $\theta$, $\beta$, $\widehat{\delta}$, $|w^0|^{(2)}_{(\underline{x},\overline{x})}$.}

{}{Indeed, for $n=1$}, recalling that $ \alpha_1,w^0\in C^{1}([\underline{x},\overline{x}])$ and $ \big|\int_{\underline{x}}^{\overline{x}}w^{0}\text{d}x\big|>0$, there exists a constant $g_0>0$ {}{depending only on $a,b,k_0,y_d,P,\eta,\beta$, $\|\alpha_1\|_\infty$, $\max_{[\underline{x},\overline{x}]}|w^0|$ and $\big|\int_{\underline{x}}^{\overline{x}}w^{0}\text{d}x\big|$} such that $|G_{w^0}(t)|\leq g_0$ for any $t\geq 0$. For any $w$ and $p$, we have
\begin{align}
-w{}{b_1}(x,t,w,p)\leq &\bigg( \|\alpha_{1x}\|_{\infty}+\frac{\|\alpha_1\|_{\infty}+g_0+1}{2} \bigg)w^2+\frac{\|\alpha_1\|_{\infty}+g_0}{2}p^2+\frac{\|\delta_1\|^2_{\infty}}{2}\notag\\
\leq &\bigg( \|\alpha_{1x}\|_{\infty}+\frac{\|\alpha_1\|_{\infty}+g_0+1}{2} \bigg)w^2
+\frac{\|\alpha_1\|_{\infty}+g_0}{2}p^2+\frac{1+\|\delta\|^2_{\infty}}{2}\notag\\
:=&c_{10}p^2+c_{11}w^2+c_{12},\notag\\
-w{}{\psi_1}(x,t,w,p)\leq &(\|\alpha_1\|_{\infty}+g_0 ):=c_{13}w^2.\notag
\end{align}
{}{Then all the structural conditions of \cite[Theorem 7.4, Chap. V]{Ladyzenskaja:1968} are satisfied. Therefore,} \eqref{+35'} has a unique solution $w_1\in \mathcal {H}^{2+\theta,1+\frac{\theta}{2}}(\overline{Q}_{T})$. Moreover, by \cite[Theorem 7.3, 7.2, Chap. V]{Ladyzenskaja:1968}, we have the following estimates:
\begin{subequations}
\begin{align}\label{+34}
\max_{\overline{Q}_{T}}|w_1|\leq& \lambda_{11}\text{e}^{\lambda_1T}\max\{\sqrt{c_{12}}, w^0(\overline{x}), w^0(\underline{x})\}:=m_1,\\
\max_{\overline{Q}_{T}}|w_{1x}|\leq &M_{1},
\end{align}
\end{subequations}
where $ \lambda_{11},\lambda_1$ depend only on $\beta, c_{10}, c_{11}$, and $c_{13}$, $ M_{1}$ depends only on $ \beta,\overline{x},\underline{x},m_1,g_0,\|\alpha_1\|_\infty,\|\alpha_{1x}\|_\infty$, and $|w^0|^{(2)}_{(\underline{x},\overline{x})}$.
Furthermore, by \cite[Theorem 5.4, Chap. V]{Ladyzenskaja:1968}, we have
\begin{align}
|w_1|^{(2+\theta)}_{Q_{T}}\leq \overline{M}_1,
\end{align}
where $ \overline{M}_1$ depends only on $ M_1,m_1,\theta,\beta,\widehat{\delta}$, and $|w^0|^{(2)}_{(\underline{x},\overline{x})}$.

{}{Assuming that for $n=k$ \eqref{+35'} has a unique solution $w_n\in \mathcal {H}^{2+\theta,1+\frac{\theta}{2}}(\overline{Q}_{T})$ satisfying \eqref{+34'''}, we prove the claim for $n=k+1$, where $k>1$ is an integer.} Noting that by \eqref{+35}, for all $t\in [0,T]$, it follows
\begin{align}\label{total}
\int_{\underline{x}}^{\overline{x}}{}{w_k}(x,t)\text{d}x=\int_{\underline{x}}^{\overline{x}}w^0(x)
\text{d}x+\int_{0}^t\int_{\underline{x}}^{\overline{x}}{}{\delta_k}(x,t)\text{d}x\text{d}t.
\end{align}
By the definition of ${}{G_{w_{k}}}(t)$, \eqref{+28} and  \eqref{+34'''}, it follows
\begin{align}\label{g_n}
|{}{G_{w_k}}(t)|\leq {}{g_k},\forall t\in [0,T],
\end{align}
where {}{$g_k$} is a positive constant depending only on $T,\delta_0$, $a,b,\overline{x},\underline{x}, \frac{k_0\eta}{P},y_d,\beta,\|\alpha\|_\infty$, {}{$\max_{\overline{Q}_{T}}|w_{k}|,\big|\int_{\underline{x}}^{\overline{x}}w^{0}\text{d}x\big|$}.
{}{
Then for any $w$ and $p$, we have
\begin{align}
-w{b_{k+1}}(x,t,w,p)\leq &c_{(k+1)0}p^2+c_{(k+1)1}w^2+c_{(k+1)2},\notag\\
-w{\psi_{k+1}}(x,t,w,p)\leq &(\|\alpha_1\|_{\infty}+g_k ):=c_{(k+1)3}w^2,\notag
\end{align}
where $c_{(k+1)0},c_{(k+1)1},c_{(k+1)2}$, and $c_{(k+1)3}$ are positive constants depending only on $g_k,\|\alpha_1\|_\infty$, $\|\alpha_{1x}\|_\infty$, and $\|\delta\|_\infty$.}

{}{By \cite[Theorem 7.4, 7.3, 7.2, Theorem 5.4, Chap. V]{Ladyzenskaja:1968},} we conclude that \eqref{+35'} has a unique solution $w_n\in \mathcal {H}^{2+\theta,1+\frac{\theta}{2}}(\overline{Q}_{T})$ for {}{$n=k+1$}, having the estimates {}{in \eqref{+34'''}. Therefore, \eqref{+35} has a unique solution $w_n\in \mathcal {H}^{2+\theta,1+\frac{\theta}{2}}(\overline{Q}_{T})$ for every {$n\geq 1$}}.

We have then that for any $n$, there exists $w_n$ satisfying \eqref{+35}. Recalling the definition of $G_{w_n}(t)$, $G_{w_n}(t)$ is fixed.  Therefore \eqref{+v} is a linear equation. Finally, the existence of a unique solution $v_n\in \mathcal {H}^{2+\theta,1+\frac{\theta}{2}}(\overline{Q}_{T})$ to \eqref{+v} is guaranteed by \cite[Theorem 7.4, Chap. V]{Ladyzenskaja:1968}. Moreover, $v_n$ has similar estimates as \eqref{+34'''}, i.e.,
\begin{align*}
\max_{\overline{Q}_{T}}|v_n|+\max_{\overline{Q}_{T}}|v_{nx}|+|v_n|^{(2+\theta)}_{Q_{T}}\leq C_n,
\end{align*}
with a certain positive constant $C_n$.
\end{proof}

\textbf{Step 2:} We prove a property of $(w_n,v_n)$ stated below.
\begin{lemma} Let $w_n,v_n\in \mathcal {H}^{2+\theta,1+\frac{\theta}{2}}(\overline{Q}_{T})$ be the solutions to \eqref{+35} and \eqref{+v}  respectively. For any $n$ and any $t\in [0,T]$, the following equalities hold:
\begin{align*}
&\int_{\underline{x}}^{\overline{x}}w_n(x,t)\text{d}x=\int_{\underline{x}}^{\overline{x}}w^0(x)
\text{d}x+\int_{0}^t\int_{\underline{x}}^{\overline{x}}\delta_n(x,s)\text{d}x\text{d}s\geq \frac{ \delta_0}{2},\\
&\int_{\underline{x}}^{\overline{x}}v_n(x,t)\text{d}x=\int_{\underline{x}}^{\overline{x}}v^0(x)
\text{d}x-\int_{0}^t\int_{\underline{x}}^{\overline{x}}\delta_n(x,s)\text{d}x\text{d}s\geq \frac{ \delta_0}{2}.
\end{align*}
\end{lemma}

\begin{proof}
It suffices to integrating over $ (\underline{x},\overline{x})$ and using the boundary conditions and \eqref{+28'}.
\end{proof}

\textbf{Step 3:} We prove the following property of $(w_n,v_n)$ which implies that $(w_n,v_n)$ is uniformly bounded in $L^1$-norm.
\begin{lemma}\label{Proposition: G_n} Let $w_n,v_n\in \mathcal {H}^{2+\theta,1+\frac{\theta}{2}}(\overline{Q}_{T})$ be the solutions to \eqref{+35} and \eqref{+v}  respectively. For any $n$ and any $t\in [0,T]$, the following estimates hold:
\begin{enumerate}
\item [(i)]
$\|w_n(\cdot,t)\|_1\leq  \|w^0\|_1
+\displaystyle\int_0^t\int_{\underline{x}}^{\overline{x}}\delta_n(x,s) \text{sgn} (w_n) \text{d}x\text{d}s,$
\item [(ii)]$\|v_n(\cdot,t)\|_1\leq  \|v^0\|_1
-\displaystyle\int_0^t\int_{\underline{x}}^{\overline{x}}\delta_n(x,s) \text{sgn} (v_n) \text{d}x\text{d}s.$
\end{enumerate}
\end{lemma}

\begin{proof}
We only prove the first inequality in the lemma. The second inequality can be assessed similarly. For $\varepsilon>0$, define
\begin{align*}
\rho_\varepsilon(r)=\left\{\begin{aligned}
& |r|,\  |r|>\varepsilon,\\
& -\frac{r^4}{8\varepsilon^3}+\frac{r^2}{4\varepsilon}+\frac{3\varepsilon}{8},\  |r|\leq \varepsilon,
\end{aligned}\right.
\end{align*}
which is $C^2$-continuous in $r$ satisfying $\rho_\varepsilon(r)\geq |r|$, $|\rho_\varepsilon'(r)|\leq 1$, and $\rho_\varepsilon''(r)\geq 0$. Let $\widetilde{G}_{n}(x,t)=\alpha_1(x)+G_{w_n}(t)$.
By \eqref{g_n}, \eqref{total}, \eqref{+28'}, we have that
\begin{align}\label{g_n'}
|\widetilde{G}_{n-1}(x,t)|\leq &\|\alpha_1\|_{\infty}+g_{n-1}\notag\\
\leq &\frac{{}{2(\|\alpha_1\|_{\infty}+g_{n-1})}}{\delta_0}\int_{\underline{x}}^{\overline{x}}|w_{n}|\text{d}x,
\end{align}
for any $(x,t)\in Q_T$.
Multiplying $\rho_\varepsilon'(w_n) $ to \eqref{+35}, and integrating
by
parts, we have
\begin{align}\label{-15}
\frac{\text{d}}{\text{d}t}\int_{\underline{x}}^{\overline{x}}\rho_\varepsilon(w_n)\text{d}x=&-\beta \|w_{nx}\sqrt{\rho_\varepsilon''(w_n)}\|^2+\int_{\underline{x}}^{\overline{x}}w_n\widetilde{G}_{n-1}(x,t)\rho_\varepsilon''(w_n)w_{nx}\text{d}x+\int_{\underline{x}}^{\overline{x}}\delta_n(x,t) \rho_\varepsilon'(w_n)\text{d}x.
\end{align}
We estimate the second term on the right-hand-side of \eqref{-15}. By Young's inequality, it follows
\begin{align}\label{-16}
\int_{\underline{x}}^{\overline{x}}w_n\widetilde{G}_{n-1}(x,t)\rho_\varepsilon''(w_n)w_{nx}\text{d}x \leq & \frac{1}{4\beta} \|w_n\widetilde{G}_{n-1}\sqrt{\rho_\varepsilon''(w_n)}\|^2
 +\beta\|w_{nx}\sqrt{\rho_\varepsilon''(w_n)}\|^2\notag\\
\leq &\frac{1}{4\beta} \|\widetilde{G}_{n-1}\|_{\infty}^2 \|w_n\sqrt{\rho_\varepsilon''(w_n)}\|^2
 +\beta\|w_{nx}\sqrt{\rho_\varepsilon''(w_n)}\|^2.
\end{align}
By \eqref{-15} and \eqref{-16}, we have
\begin{align}\label{-17}
\frac{\text{d}}{\text{d}t}\int_{\underline{x}}^{\overline{x}}\rho_\varepsilon(w_n)\text{d}x\leq&\frac{1}{4\beta} \|\widetilde{G}_{n-1}\|_{\infty}^2 \|w_n\sqrt{\rho_\varepsilon''(w_n)}\|^2+\int_{\underline{x}}^{\overline{x}}\delta_n(x,t) \rho_\varepsilon'(w_n)\text{d}x.
\end{align}
Note that
\begin{align}\label{-20}
\|w_n\sqrt{\rho_\varepsilon''(w_n)}\|^2=&\int_{\underline{x}}^{\overline{x}}w_n^2\rho_\varepsilon''(w_n)\chi_{\{|w_n|>\varepsilon\}}(x)\text{d}x
+\int_{\underline{x}}^{\overline{x}}w_n^2\rho_\varepsilon''(w_n)\chi_{\{|w_n|\leq\varepsilon\}}(x)\text{d}x\notag\\
=&\int_{\underline{x}}^{\overline{x}}w_n^2\rho_\varepsilon''(w_n)\chi_{\{|w_n|\leq\varepsilon\}}(x)\text{d}x\notag\\
=&\int_{\underline{x}}^{\overline{x}}w_n^2\frac{3(\varepsilon^2-w_n^2)}{2\varepsilon^3}\chi_{\{|w_n|\leq\varepsilon\}}(x)\text{d}x\notag\\
\leq&\frac{3}{2}\varepsilon( \overline{x}-\underline{x}).
\end{align}
By \eqref{g_n'}, \eqref{-17}, and \eqref{-20}, it follows
\begin{align}\label{-21}
\frac{\text{d}}{\text{d}t}\int_{\underline{x}}^{\overline{x}}\rho_\varepsilon(w_n)\text{d}x \leq&\frac{3\varepsilon( \overline{x}-\underline{x})}{{}{4}\beta}
\frac{\|\alpha_1\|_{\infty}+g_{n-1}}{\delta_0}
\bigg(\int_{\underline{x}}^{\overline{x}}|w_{n}|\text{d}x\bigg)^2
 +\int_{\underline{x}}^{\overline{x}}\delta_n(x,t)  \rho_\varepsilon'(w_n)\text{d}x\notag\\
\leq& \frac{3\varepsilon( \overline{x}-\underline{x})}{{}{4}\beta}
\frac{\|\alpha_1\|_{\infty}+g_{n-1}}{\delta_0}\bigg(\int_{\underline{x}}^{\overline{x}}\rho_\varepsilon(w_{n})\text{d}x\bigg)^2
 +\int_{\underline{x}}^{\overline{x}}\delta_n(x,t)\rho_\varepsilon'(w_n) \text{d}x\notag\\
:=&c_n\varepsilon\bigg(\int_{\underline{x}}^{\overline{x}}\rho_\varepsilon(w_{n})\text{d}x\bigg)^2
  +\int_{\underline{x}}^{\overline{x}}\delta_n(x,t) \rho_\varepsilon'(w_n) \text{d}x.
\end{align}
We infer from Gronwall's inequality:
\begin{align}
\int_{\underline{x}}^{\overline{x}}\rho_\varepsilon(w_n)\text{d}x
\leq&  \int_{\underline{x}}^{\overline{x}}\rho_\varepsilon(w^0)\text{d}x
\times \text{e}^{c_n\varepsilon\int_0^t\int_{\underline{x}}^{\overline{x}}
\rho_\varepsilon(w_n(x,s))\text{d}x\text{d}s}
\notag\\
&+\int_0^t\int_{\underline{x}}^{\overline{x}}{\delta_n(x,s)} \rho_\varepsilon'(w_n{(x,s)}) \text{d}x \times \text{e}^{c_n\varepsilon\int_s^t\int_{\underline{x}}^{\overline{x}}
\rho_\varepsilon(w_n(x,\tau))\text{d}x\text{d}\tau}\text{d}s.
\end{align}
Note that $ w_n$ is smooth and bounded on $[\underline{x},\overline{x}]\times[0,T]$. Letting $\varepsilon\rightarrow 0$, we get
\begin{align*}
      \int_{\underline{x}}^{\overline{x}}|w_n(x,t)|\text{d}x 
\leq & \int_{\underline{x}}^{\overline{x}}|w^0(x)|\text{d}x
       +\int_0^t\int_{\underline{x}}^{\overline{x}}{\delta_n(x,s)} \text{sgn} (w_n) \text{d}x\text{d}s.
\end{align*}
%
\end{proof}

\textbf{Step 4:} We prove that every constant of the estimates of $w_n$ and $v_n$ in the proof of Lemma~\ref{Proposition: G_n} are independent of $n$, i.e., we have the result stated below.
\begin{lemma}\label{Lemma G_n} Let $w_n,v_n\in \mathcal {H}^{2+\theta,1+\frac{\theta}{2}}(\overline{Q}_{T})$ be the solutions to \eqref{+35} and \eqref{+v}  respectively. For any $n$, it holds:
\begin{subequations}
\begin{align}
&\max_{\overline{Q}_{T}}|w_n|\leq {m},\max_{\overline{Q}_{T}}|w_{nx}|\leq M,|w_n|^{(2+\theta)}_{Q_{T}}\leq \overline{M},\label{44a}\\
&\max_{\overline{Q}_{T}}|v_n|\leq {l},\max_{\overline{Q}_{T}}|v_{nx}|\leq {L},|v_n|^{(2+\theta)}_{Q_{T}}\leq \overline{L},\label{44b}
\end{align}
\end{subequations}
where {$m,M,\overline{M},\widetilde{M},l,L$, and $\overline{L}$} are positive constants independent of $n$.
\end{lemma}

\begin{proof}
By the proof of Lemma \ref{Proposition: G_n} and $\|\delta_n\|_\infty\leq 1+\|\delta\|_\infty$, $\|w_n(\cdot,t)\|_1$ is uniformly bounded in $n$. Then by the continuity of $w_n$, we deduce that $w_n$ is uniformly bounded in $n$ on $\overline{Q}_{T}$. Note that $\int_{\underline{x}}^{\overline{x}}\beta w_{nx}\text{d}x=\beta(w_{n}(\overline{x},t)-w_{n}(\underline{x},t))$. By \eqref{total}, \eqref{+28}, Lemma~\ref{Proposition: G_n}, and the a.e. convergence of $\delta_n$ to $\delta$, we have that $G_{w_n}(t)$ is uniformly bounded in $n$ on $\overline{Q}_{T}$. Then {}{$m_n,M_n,\overline{M}_n$ in \eqref{+34'''}} are all independent of $n$. Therefore {}{\eqref{44a} holds. \eqref{44b} is a consequence} of the uniform boundedness of $G_{w_n}(t)$ and \cite[Theorem 5.4, 7.2-7.4, Chap. V]{Ladyzenskaja:1968}.
\end{proof}

{{}
\subsection{Proof of Theorem~\ref{Proposition 9}}\label{App: uniqueness}
}
Note that as \eqref{eq: foced FPE OFF} is linear in $v$, it suffices to prove that $w_1=w_2$ in $Q_T$. Letting $\widetilde{w}=w_1-w_2$, it follows
\begin{align*}
 &\dfrac{\partial \widetilde{w}}{\partial t}= \dfrac{\partial}{\partial x}\bigg( \beta\dfrac{\partial \widetilde{w}}{\partial x}- \alpha_{1}(x)\widetilde{w}-\widetilde{w}G_{w_2}(t) - w_1(G_{w_1}(t)-G_{w_2}(t) )\bigg),\\
&\beta\dfrac{\partial \widetilde{w}}{\partial x}(\underline{x},t)- \alpha_{1}(\underline{x})\widetilde{w}(\underline{x},t)-\widetilde{w}(\underline{x},t)G_{w_2}(t) - w_1(\underline{x},t)(G_{w_1}(t)-G_{w_2}(t) )=0,\\
 &\beta\dfrac{\partial \widetilde{w}}{\partial x}(\overline{x},t)- \alpha_{1}(\overline{x})\widetilde{w}(\overline{x},t)-\widetilde{w}(\overline{x},t)G_{w_2}(t) - w_1(\overline{x},t)(G_{w_1}(t)-G_{w_2}(t) )=0,\\
 &\widetilde{w}(x,0)=0.
\end{align*}
Then we have
\begin{align}
 \dfrac{1}{2}\frac{\text{d}}{\text{d}t}\|\widetilde{w}\|^2=&-\int_{\underline{x}}^{\overline{x}}\dfrac{\partial \widetilde{w}}{\partial x}\bigg( \beta\dfrac{\partial \widetilde{w}}{\partial x}- \alpha_{1}(x)\widetilde{w}-\widetilde{w}G_{w_2}(t) \bigg)\text{d}x+\int_{\underline{x}}^{\overline{x}}\dfrac{\partial \widetilde{w}}{\partial x}w_1(G_{w_1}(t)-G_{w_2}(t) ) \text{d}x.\label{+47}
\end{align}
Note that there exist positive constants $C_0$ and $C_1$ such that
\begin{align}
|\alpha_1(x)+G_{w_2}(t)|\leq C_0, \forall (x,t)\in \overline{Q}_{T},\label{+48}
\end{align}
and
\begin{align}
\bigg|G_{w_1}(t)-G_{w_2}(t)\bigg|=\frac{\bigg|\int_{\underline{x}}^{\overline{x}}\big(\alpha_1(x) +{}{\frac{k_0}{a}}(ax+b)\big)\widetilde{w}\text{d}x\bigg|}{\bigg|\int_{\underline{x}}^{\overline{x}}w_2\text{d}x\bigg|}\leq C_1\|\widetilde{w}\|_1,\label{+49}
\end{align}
where we used $\widetilde{w}(\overline{x},t)=\widetilde{w}(\underline{x},t)$, H\"{o}lder's inequality, and the fact that by Proposition~\ref{prop. 8}
\begin{align*}
\bigg|\int_{\underline{x}}^{\overline{x}}w_1\text{d}x\bigg|=\bigg|\int_{\underline{x}}^{\overline{x}}w_2\text{d}x\bigg|
=\bigg| \int_{\underline{x}}^{\overline{x}}w^0(x)\text{d}x +\int_{0}^t\int_{\underline{x}}^{\overline{x}}{}{\delta(x,s)}\text{d}x{}{\text{d}s}\bigg|
\geq \delta_0.
\end{align*}
By \eqref{+47}, \eqref{+48}, \eqref{+49}, H\"{o}lder's inequality, and Young's inequality, we have
\begin{align*}
\dfrac{1}{2}\frac{\text{d}}{\text{d}t}\|\widetilde{w}\|^2\leq&-\beta\|\widetilde{w}_x\|^2+\frac{\varepsilon}{2}\|\widetilde{w}_x\|^2+\frac{\varepsilon}{2}\|\widetilde{w}_x\|^2
+\frac{C_0^2}{2\varepsilon}\|\widetilde{w}\|^2+\frac{\varepsilon}{2}\|\widetilde{w}_x\|^2+\frac{1}{2\varepsilon}\|w_1\|^2 \cdot C_1^2 \|\widetilde{w}\|_1^2\\
\leq & -\left(\beta-\frac{3\varepsilon}{2}\right)\|\widetilde{w}_x\|^2+C_2\|\widetilde{w}\|^2,
\end{align*}
where $\varepsilon>0$, $C_2$ is a positive constant depending only on $\overline{x},\underline{x},\|w_1\|,C_0,C_1$, and $\varepsilon $.
Choosing $\varepsilon=\frac{2}{3}\beta$, it follows
\begin{align*}
\frac{\text{d}}{\text{d}t}\|\widetilde{w}\|^2\leq C_3\|\widetilde{w}\|^2,
\end{align*}
where $C_3$ is a positive constant.
By Gronwall's inequality, we have
\begin{align*}
\|\widetilde{w}(\cdot,t)\|^2\leq \|\widetilde{w}(\cdot,0)\|^2\cdot \text{e}^{C_3t}=0,
\end{align*}
based on which and by the continuity of $\widetilde{w}$, it yields $\widetilde{w}\equiv 0$ on $\overline{Q}_{T}$.


\end{document}